\documentclass[twoside,11pt]{article}
\usepackage{graphicx}
\usepackage{color}
\usepackage{amsmath,amssymb,amsthm,amsfonts}
\usepackage{graphicx}
\usepackage{color}
\usepackage{multicol}

\usepackage{mathptmx}
\usepackage{enumerate}
\usepackage[numbers,sort&compress]{natbib}
\numberwithin{equation}{section}
\newcommand{\R}{{\mathbb R}}

\newcommand{\be}{\begin{equation}}
\newcommand{\ee}{\end{equation}}

\newcommand{\ben}{\begin{eqnarray*}}
\newcommand{\enn}{\end{eqnarray*}}

\newcommand{\va}{\varepsilon}
\newcommand{\ti}{\tilde}

\newtheorem{theorem}{\textbf Theorem}[section]
\newtheorem{lemma}{\textbf Lemma}[section]

\newtheorem{prop}{\textbf Proposition}[section]
\newtheorem{defin}{\textbf Definition}[section]

\def\endProof{{\hfill$\Box$}}

\begin{document}
\title{{\textbf{Radial boundary layers for the singular Keller-Segel model} }}
\author{Qianqian Hou\thanks{Institute for Advanced Study in Mathematics, Harbin Institute of Technology, Harbin 150001, People's Republic of China
({\tt qianqian.hou@connect.polyu.hk}).}}
\date{}
\maketitle

\begin{quote}
{\sc Abstract}
This paper is concerned with the diffusion limit (as $\va\rightarrow 0$) of radial solutions to a chemotaxis system with logarithmic singular sensitivity in a bounded interval with mixed Dirichlet and Robin boundary conditions. We use a Cole-Hopf type transformation  to resolve the logarithmic singularity and prove that the solution of the transformed system has a boundary-layer profile as $\va \to 0$, where the boundary layer thickness is of $\mathcal{O}(\va^{\alpha})$ with $0<\alpha<\frac{1}{2}$. By transferring the results back to the  original chemotaxis model via Cole-Hopf transformation, we find that boundary layer profile is present at the gradient of solutions and the solution itself is uniformly convergent with respect to  $\va>0$.



\end{quote}
\maketitle

\section{Introduction}
Chemotaxis describes the oriented movement of species stimulated by uneven distribution of a chemical substance in the environment.
It is a significant mechanism accounting for abundant biological process/phenomenon, such as aggregation of bacteria \cite{Murray1, TLM2}, slime mould formation \cite{Hofer1}, fish pigmentation \cite{PainterFish}, tumor angiogenesis \cite{Chaplain93, CPZ1, CPZ3}, primitive streak formation \cite{PainterStreak}, blood vessel formation \cite{Gamba03}, wound healing \cite{Pettet96}. Mathematical models of chemotaxis were first proposed by Keller and Segel in their seminal works \cite{KS1, KS71a, KS71b}. In this paper, we are concerned with the following chemotaxis model:
\begin{eqnarray}\label{ks1}
\left\{\begin{array}{lll}
u_t= \nabla\cdot[D\nabla u-  \chi u\nabla (\ln c)], \quad (x,t)\in \Omega\times (0,\infty)\\
c_t=\varepsilon \Delta c-\mu u c,
\end{array}\right.
\end{eqnarray}
where $\Omega$ is a domain in $\mathbb{R}^n$ with smooth boundary. System \eqref{ks1} was first advocated in \cite{KS71b} to describe the traveling band propagation of bacterial chemotaxis observed in the experiment of Adler \cite{Adler66, Adler69}. It later appeared in the work by Levine et al \cite{LSN} to model the initiation of tumor angiogenesis, where
$u(x,t)$ represents the density of vascular endothelial cells and $c(x,t)$ denotes the concentration of signaling molecules vascular endothelial growth factor (VEGF). The parameters $D>0$, $\va\geq 0$ are diffusion coefficients of the endothelial cells and the chemical VEGF respectively, $\chi>0$ is the chemotactic coefficient measuring the intensity of chemotaxis and $\mu\geq 0$ is the chemical consumption rate by cells. In particular it was pointed out in \cite{LSN} that the chemical diffusion process is far less important comparing to its interaction with endothelial cells and thus the diffusion coefficient $\va$ could be small or negligible.
Despite of its biological significance, \eqref{ks1} is difficult to study mathematically due to the singularity of $\ln c$ at $c=0$. The well-known way to overcome this singularity was applying the following Cole-Hopf transformation
(cf. \cite{LeSl, LW11}):
\be\label{ch}
\vec{v}=-\nabla \ln c=-\frac
{\nabla c}{c}
\ee
to transform \eqref{ks1} into a system of conservation laws:
\be\label{hs}
    \left\{ \begin{array}{ll}
    u_t - \chi\nabla\cdot ( u \vec{v}) = D\Delta u, \\[1mm]
   \vec{v}_t -\nabla(u-\va|\vec{v}|^2 ) = \va \Delta \vec{v},\\
   (u, \vec{v})(x,0)=(u_0, \vec{v}_0)(x).
    \end{array} \right.
\ee

The transformed system \eqref{hs} attracts extensive attentions and numerous interesting results have been developed. We briefly recall these results by the dimension of spaces. In the one dimensional case, the global well-posedness along with large time behavior of solutions was investigated when $\Omega=\mathbb{R}$ in \cite{GXZZ,Li-Pan-Zhao} with $\va=0$ and in \cite{peng-ruan-zhu2012global,MWZ} with $\va>0$.
 When $\Omega=(0,1)$, authors in \cite{zhang-zhu2007,LPZ} obtained the unique global solution under Neumann-Dirichlet boundary conditions for $\va=0$, and the result was later extended to the case $\va>0$ in \cite{Wang-Zhao13,TWW}. The problem with $\va\geq 0$ is also globally well-posed \cite{LZ} with Dirichlet-Dirichlet boundary conditions. Furthermore, the existence and stability of traveling wave solutions were studied in \cite{JLW, LW09, LW10, LW11, LW12, LLW, Chae}. However to the best of our knowledge, except when it is associated with radially symmetric initial data, the known well-posedness results of problem  \eqref{hs} in the multi-dimension are merely confined to local large and global small solutions, cf. \cite{LLZ,Hao,DL,PWZ, WWZ, WXY} for details when $\Omega=\mathbb{R}^n$ ($n\geq 2$) and \cite{LPZ, rebholz2019initial} when $\Omega \subset \mathbb{R}^n$ ($n\geq 2$) is bounded.
 If the initial data are radially symmetric and $\Omega=B_{R}(0):=\{x=(x_1,x_2,\cdots,x_n)\in \mathbb{R}^{n}\,|\,\, |x|< R\}$ ($n\geq 2$), Winkler \cite{winkler2017radial} recently proved that \eqref{ks1} with $\va>0$ subject to Neumann boundary conditions 
admits a global generalized (weak) solution which is radially symmetric and smooth away from the origin $x=0$.

 In addition to the above well-posedness results, the asymptotic behavior of solutions as $\va\rightarrow 0$ is a particularly relevant issue
 (mentioned in \cite{LSN} that $\va$ could be small/negligible) and has been studied in several circumstances for equation \eqref{hs} and thus for the original equation \eqref{ks1} via transformation \eqref{ch}. For illustration, denote by $(u^\va,\vec{v}^{\,\va})$ and $(u^0,\vec{v}^{\,0})$ the solutions of \eqref{hs} with $\va>0$ and $\va=0$ respectively (when $n=1$ we shall use the notation $v$ instead of $\vec{v}$ since it is a scalar). First in unbounded domains, it has been shown that both traveling wave solutions (cf. \cite{Wang-DCDSB-review}) in $\R$ and the global small-data solution of the Cauchy problem (cf. \cite{WXY, PWZ}) in $\R^n(n=2,3)$ are uniformly convergent in $\varepsilon$, namely $(u^\va,\vec{v}^{\,\va})$ converge to $(u^0,\vec{v}^{\,0})$ in $L^\infty$-norm as $\varepsilon \to 0$.
  With $\Omega=(0,1)$, the solutions still converge (cf. \cite{Wang-Zhao13}) as $\va\rightarrow 0$ when \eqref{hs} is endowed with the following mixed homogeneous Neumann-Dirichlet boundary conditions
$$u_x|_{x=0,1}=v|_{x=0,1}=0.
$$
However if both of $u$ and $v$ are subject to the Dirichlet boundary conditions, one can not preassign the boundary value for $v^0$ when $\va=0$ since it is intrinsically determined by the second equation of \eqref{hs} as $v^0|_{x=0,1}=v_0|_{x=0,1}+\int_{0}^t u_x^0|_{x=0,1}\,d\tau $. Thus the plausible Dirichlet boundary conditions should be prescribed as:
\begin{equation}\label{bc}
\left\{\aligned
&u|_{x=0,1}=\bar{u}\geq 0,\quad v|_{x=0,1}=\bar{v}, \ & \mathrm{if}\ \varepsilon>0,\\
&u|_{x=0,1}=\bar{u}\geq 0, \ & \ \ \ \ \mathrm{if} \ \varepsilon=0,
\endaligned\right.
\end{equation}
where $\bar{u}\geq 0,$ $\bar{v}\in \mathbb{R}$ are constants.
  In this case, if the boundary values of $v$ with $\va>0$ and $\va=0$ do not match, then the solution component $v$ would diverge near the end points $x=0,1$ as $\va\rightarrow 0$ and this phenomenon is termed as the boundary layer effect, which has been an important topic in the fluid mechanics \cite{S} when investigating the inviscid limit of the Navier-Stokes equations near a boundary and has attracted extensive studies (cf. \cite{F, Frid99, Frid20, JZ, WX, XY, YZZ})  since the pioneering work \cite{P} by Prandtl in 1904. In particular, this boundary layer effect for problem \eqref{hs}-\eqref{bc} has been recently numerically verified in \cite{LZ} and rigorously proved in \cite{HWZ}.
Enlightened by these results, it is natural to expect that \eqref{hs} in multi-dimension ($n\geq 2$) possesses boundary layer solutions as well when prescribing appropriate Dirichlet boundary conditions. In particular we aim to investigate this issue for its radial solutions in the present paper. To this end, we  first rewrite \eqref{ks1} in its radially symmetric form by assuming that the solutions $(u,c)$ are radially symmetric, depending only on the radial variable $r=|x|$ and time variable $t$. In a domain bounded by two concentric sphere, i.e. $\Omega=\{x=(x_1,x_2,\cdots,x_n)\in \mathbb{R}^n\,|\,\,\, 0<a<|x|<b
 \}$, \eqref{ks1} reads as
\begin{eqnarray}\label{e1}
\left\{\begin{array}{lll}
u_t= \frac{1}{r^{n-1}}(r^{n-1}u_{r})_r
-\frac{1}{r^{n-1}}\left(r^{n-1}u\frac{c_r}{c}\right)_r
,\quad (r,t)\in (a,b)\times (0,\infty)\\
c_t=\va \frac{1}{r^{n-1}}(r^{n-1}c_{r})_r-uc,\\
u(r,0)=u_0(r),\quad c(r,0)=c_0(r),
\end{array}\right.
\end{eqnarray}
where $D=\chi=\mu=1$ have been assumed without loss of generality.
 Similar as deriving \eqref{hs} from \eqref{ks1}, we apply the following Cole-Hopf type transformation
\be\label{e2}
v=-(\ln c)_{r}=-\frac{c_r}{c},
\ee
which turns \eqref{e1} into
\begin{eqnarray}\label{e3}
\left\{\begin{array}{lll}
u_t= \frac{1}{r^{n-1}}(r^{n-1}u_{r})_r
+\frac{1}{r^{n-1}}(r^{n-1}uv)_r, \qquad\,(r,t)\in (a,b)\times (0,\infty)\\
v_t=\va\frac{1}{r^{n-1}} (r^{n-1}v_{r})_r-\va\frac{n-1}{r^2}v-\va(v^2)_r+u_r,\\
(u,v)(r,0)=(u_0,v_0)(r).
\end{array}\right.
\end{eqnarray}
 Similar to \eqref{bc}, the Dirichlet boundary conditions for \eqref{e3} are prescribed as
\be\label{e4}
\left\{\begin{array}{lll}
u|_{r=a,b}=\bar{u},\,\,
v|_{r=a}=\bar{v}_1,\,v|_{r=b}=\bar{v}_2,\quad {\rm if} \,\, \va>0,\\
u|_{r=a,b}=\bar{u},\quad \qquad \qquad \qquad \qquad \quad\, \,\,{\rm if} \, \,\va=0.
\end{array}\right.
\ee

In this paper, we shall investigate  the asymptotic behavior of solutions to \eqref{e3}-\eqref{e4} as $\va\rightarrow 0$ for $n\geq 2$ (if $n=1$, it coincides with the one-dimensional model \eqref{hs}-\eqref{bc} which has been studied in \cite{HWZ} as aforementioned). In particular, the solution component $v$ is proved to have a boundary layer due to the mismatch of its boundary values as $\va\rightarrow 0$ (see Theorem \ref{t1}).

\section{Main results}
To study the boundary layer effect, we first present the global well-posedness and regularity estimates for solutions of \eqref{e3}-\eqref{e4} with $\va=0$ in Theorem \ref{p1}. By these estimates, we then state the main result on the convergence for $u$ and boundary layer formation by $v$ in Theorem \ref{t1}. Finally, the result is converted to the original chemotaxis model \eqref{e1} via \eqref{e2}. We begin with introducing some notations.\\
\newline
\textbf{Notations.} Without loss of generality, we assume $0\leq \va<1$ since the zero diffusion limit as $\va\rightarrow 0$ is our main concern. Throughout this paper, unless specified, we use $C$ to denote a generic positive constant which is independent of $\va$ and dependent on $T$. In contrast, $C_0$ denotes a generic constant independent of $\va$ and $T$. For simplicity, $L^p$ represents $L^p(a,b)$ with $1\leq p\leq \infty$, $H^k$ denotes $H^{k}(a,b)$ with $k\in \mathbb{N}$ and $\|\cdot\|$ stands for $\|\cdot\|_{L^2}$. Moreover, if $f(r,t)\in L^p(a,b)$ for fixed $t>0$, we use $\|f(t)\|_{L^p}$ to denote $\|f(\cdot,t)\|_{L^p}$.\\

The first result is on the global well-posedness of \eqref{e3}-\eqref{e4} with $\va=0$.
\begin{theorem}\label{p1}
Assume that $(u_0,v_0)\in H^2\times H^2$ with $u_0\geq 0$ satisfy the compatible conditions
$u_0(a)=u_0(b)=\bar{u}$.
Then the initial-boundary value problem \eqref{e3}-\eqref{e4} with $\va=0$ has a unique solution $(u^0,v^0)\in C([0,\infty);H^2\times H^2)$  such that the following estimates hold true.\\
(i) If $\bar{u}>0$, there is a constant $C_0$ independent of $t$ such that
\be\label{k23}
\begin{split}
\|u^0(t)-\bar{u}\|_{H^2}^2+\|v^0(t)\|_{H^2}^2+\int_0^t \big(\|u^0(\tau)-\bar{u}\|_{H^3}^2
+\bar{u}\|(r^{n-1}v^0)_r(\tau)\|_{H^1}^2\big)\,d\tau\leq C_0.
\end{split}
\ee
Moreover,
\be\label{i10}
\lim_{t\rightarrow \infty}\|u^0(t)-\bar{u}\|_{L^\infty}=0.
\ee
(ii) If $\bar{u}=0$, for any $0<T<\infty$, there exists a constant $C$ depending on $T$ such that
\be\label{k26}
\|u^0\|_{L^\infty(0,T;H^2)}+\|v^0\|_{L^\infty(0,T;H^2)}
+\|u^0\|_{L^2(0,T;H^3)}\leq C.
\ee
\end{theorem}

 We proceed to recall the definition of boundary layers (BLs) following the convention of \cite{Frid99, Frid20}.
\begin{defin}\label{D1}
 Denote by $(u^\va, v^\va)$ and $(u^0, v^0)$ the solution of \eqref{e3}-\eqref{e4} with $\va>0$ and $\va=0$, respectively. If there exists a non-negative function $\delta=\delta(\va)$ satisfying $\delta(\va)\rightarrow 0$ as $\va \rightarrow0$ such that
\ben
\begin{aligned}
&\lim_{\va\rightarrow 0}\|u^\va-u^0\|_{L^{\infty}(0,T;C[a,b])}=0,\\
&\lim_{\va\rightarrow 0}\|v^{\va}-v^0\|_{L^{\infty}(0,T;C[a+\delta,b-\delta])}=0,\\
&\liminf_{\va\rightarrow 0}\|v^{\va}-v^0\|_{L^{\infty}(0,T;C[a,b])}>0,
\end{aligned}
\enn
we say that the initial-boundary value problem \eqref{e3}-\eqref{e4} has a boundary layer solution as $\va \to 0$ and $\delta(\varepsilon)$ is called a boundary layer thickness (BL-thickness).
\end{defin}

Our main result is as follows.
\begin{theorem}\label{t1}
 Suppose that $(u_0,v_0)\in H^2\times H^2$ with $u_0\geq 0$ satisfy the compatible conditions $u_0(a)=u_0(b)=\bar{u}$ and $v_0(a)=\bar{v}_1, v_0(b)=\bar{v}_2$. Let $(u^0,v^0)$ be the solution obtained in Theorem \ref{p1}. For any $0<T<\infty$, we denote
\ben
\va_0=\min \bigg\{\Big(8C_0\int_0^T F(t)\,dt\Big)^{-2}, \, \,
\Big(32C_0^2Te^{C_0\int_0^T F(t)\,dt}\int_0^T F(t)\,dt\Big)^{-2}
\bigg\},
\enn
where the function $F(t)$ is defined in \eqref{j14} by $\|u^0(t)\|_{H^2}$, $\|v^0(t)\|_{H^2}$ and the constant $C_0$  (given in \eqref{i5}) depends only on $a,b$ and $n$.
Then \eqref{e3}-\eqref{e4} with $\va\in (0,\va_0]$ admits a unique solution $(u^\va,v^\va)\in C([0,T];H^2\times H^2)$. Furthermore, any function $\delta=\delta(\va)$ satisfying
\be\label{j13}
\delta(\va)\rightarrow 0\,\,{\rm and}\,\,
\va^{1/2}/\delta(\va)\rightarrow 0,\, {\rm as}\,\va \rightarrow 0
\ee
is a BL-thickness of \eqref{e3}-\eqref{e4} such that
\be\label{j6}
\|u^\va-u^0\|_{L^\infty(0,T;C[a,b])}\leq C\va^{1/4},
\ee
\be\label{j7}
\|v^\va-v^0\|_{L^\infty(0,T;C[a+\delta,b-\delta])}\leq C\va^{1/4}\delta^{-1/2}.
\ee
Moreover,
\be\label{j8}
\liminf_{\va\rightarrow 0} \|v^\va-v^0\|_{L^\infty(0,T;C[a,b])}>0
\ee
if and only if
\be\label{j9}
\int_{0}^t u^{0}_r(a,\tau)\,d\tau\neq 0\qquad
{\rm or}\quad\int_{0}^t u^{0}_r(b,\tau)\,d\tau\neq 0,\quad
{\rm for\,\, some}\,\, t\in [0,T].
\ee
\end{theorem}


 By employing transformation \eqref{e2}, we next convert the above results for \eqref{e3}-\eqref{e4} to the pre-transformed chemotaxis model \eqref{e1}. The counterpart of the original model reads as follows:
\be\label{e6}
\left\{\begin{array}{lll}
u_t= \frac{1}{r^{n-1}}(r^{n-1}u_{r})_{r}
-\frac{1}{r^{n-1}}\left(r^{n-1}u\frac{c_r}{c}\right)_r,\\
c_t=\va\frac{1}{r^{n-1}}(r^{n-1} c_{r})_{r}-uc,\\
u(0,r)=u_0(r),\,\,\,c(0,r)=c_0(r),\\
u|_{r=a,b}=\bar{u},\,\,\,
[c_r+\bar{v}_1c](a,t)=0,\,\,
[c_r+\bar{v}_2c](b,t)=0.
\end{array}\right.
\ee
\begin{prop}\label{p2}
  Assume $c_0> 0$ and $(u_0,\ln c_0)\in H^2\times H^3$. Suppose that the assumptions in Theorem \ref{t1} hold with $v_0=-(\ln c_0)_r$. Let $0<T<\infty$. Then \eqref{e6} with $\va\in [0,\va_0]$ admits a unique solution $(u^\va,c^\va)\in C([0,T];H^2\times H^3)$ such that
\be\label{i13}
\begin{split}
\|u^\va-u^0\|_{L^\infty(0,T;C[a,b])}\leq C\va^{1/4},\\
\|c^\va-c^0\|_{L^\infty(0,T;C[a,b])}\leq C\va^{1/4}.
\end{split}
\ee
Moreover, the gradient of $c$ has a boundary layer effect as $\va\rightarrow 0$, that is
\be\label{i11}
\|c^\va_r-c^0_r\|_{L^\infty(0,T;C[a+\delta,b-\delta])}
\leq C\va^{1/4}\delta^{-1/2},
\ee
with the function $\delta(\va)$ defined \eqref{j13}
and the following estimate holds
\be\label{i12}
\liminf_{\va\rightarrow 0} \|c^\va_r-c^0_r\|_{L^\infty(0,T;C[a,b])}>0,
\ee
if and only if \eqref{j9} is true.
\end{prop}


At the end of this section, we briefly introduce the main ideas used in the paper. Although the system \eqref{e3}-\eqref{e4} with $n\geq 2$ is in a similar form to its counterpart with $n=1$ for which the vanishing diffusion limit has been studied in \cite{HWZ} based on a $\va$-independent estimate for solutions with $\va>0$, the methods used there can not be applied to study the present problem since when $n\geq 2$ the system \eqref{e3}-\eqref{e4} with $\va>0$ lacks an energy-like structure or a Lyapunov function to provide a preliminary estimate uniformly in $\va$. Moreover, one can not use the estimates derived in \cite{winkler2017radial} for the present problem either since those estimates depend on $\va$. The difficulty in our analysis consists in deriving the $\va$-convergence estimates in \eqref{j6} and \eqref{j7} without any uniform-in-$\va$ priori bounds on solutions $(u^{\va},v^{\va})$. Inspired by the works \cite{constantin1986note,wuxu2014}, this will be achieved in section 4 by regarding $(u^\va,v^\va)$ with small $\va>0$ as a perturbation of $(u^0,v^0)$ and then estimating their difference $(u^\va-u^0,v^\va-v^0)$ by the method of energy estimates and a new Gronwall's type inequality (see Lemma \ref{l0}) on ODEs. The proof of Theorem \ref{p1} is standard and will be given in section 3.

\section{Proof of Theorem \ref{p1}}
  This section is to prove Theorem \ref{p1} based on the following lemmas where the \emph{a priori} estimates on solution $(u^0,v^0)$ of \eqref{e3}-\eqref{e4} with $\va=0$ are derived by the energy method. We set off by rewriting \eqref{e3}-\eqref{e4} with $\va=0$ as follows:
\begin{eqnarray}\label{e7}
\left\{\begin{array}{lll}
u^0_t= \frac{1}{r^{n-1}}\big(r^{n-1}u^0_{r}\big)_r
+\frac{1}{r^{n-1}}
\big(r^{n-1}u^0v^0\big)_r,\\
v^0_t=u^0_r,\\
(u^0,v^0)(r,0)=(u_0,v_0)(r),\\
u^0(a,t)=u^0(b,t)=\bar{u}.
\end{array}\right.
\end{eqnarray}

\begin{lemma}\label{l01}
Suppose the assumptions in Theorem \ref{p1} hold and $\bar{u}>0$. Then there exists a positive constant $C_0$ independent of $t$ such that
\be\label{k1}
\begin{split}
\int_a^b& r^{n-1}[(u^0\ln u^0-u^0)(t)-(\bar{u}\ln \bar{u}-\bar{u})
-\ln \bar{u}(u^0(t)-\bar{u})]dr\\
+&\frac{1}{2}\int_a^b r^{n-1}(v^0)^2(t)dr+\int_0^t\int_a^b r^{n-1}\frac{(u^0_r)^2}{u^0}drd\tau\leq C_0
\end{split}
\ee
and
\be\label{k2}
\|r^{(n-1)/2}[u^0(t)-\bar{u}]\|^2
+\int_0^t\|r^{(n-1)/2}u^0_r(\tau)\|^2d\tau \leq C_0.
\ee
\end{lemma}
\begin{proof}
Taking the $L^2$ inner products of the first and second equation of \eqref{e7} with $r^{n-1}(\ln u^0-\ln\bar{u})$ and $r^{n-1}v^0$ respectively, we then add the results and use integration by parts to get
\ben
\begin{split}
\frac{d}{dt}&\int_a^b r^{n-1}[(u^0\ln u^0-u^0)-(\bar{u}\ln \bar{u}-\bar{u})
-\ln \bar{u}(u^0-\bar{u})]dr\\
+&\frac{1}{2}\frac{d}{dt}\int_a^b r^{n-1}(v^0)^2dr
+\int_a^b r^{n-1}\frac{(u^0_r)^2}{u^0}dr
=0,
\end{split}
\enn
which gives rise to \eqref{k1} upon integration over $(0,t)$. To prove \eqref{k2}, we denote $\ti{u}(r,t)=u^0(r,t)-\bar{u}$ and find from \eqref{e7} that $(\ti{u},v^0)(r,t)$ satisfies
\begin{eqnarray}\label{e8}
\left\{\begin{array}{lll}
\ti{u}_t=\frac{1}{r^{n-1}}(r^{n-1}\ti{u}_r)_r
+\frac{1}{r^{n-1}}(r^{n-1}\ti{u}v^0)_r
+\frac{\bar{u}}{r^{n-1}}(r^{n-1}v^0)_r,\\
v^0_t=\tilde{u}_r,\\
(\tilde{u},v^0)(r,0)=(u_0-\bar{u},v_0)(r),\\
\tilde{u}(a,t)=\tilde{u}(b,t)=0.
\end{array}\right.
\end{eqnarray}
Multiplying the first and second equation of \eqref{e8} by $r^{n-1}\ti{u}$ and $\bar{u}r^{n-1}v^0$, respectively. Adding the results gives
\be\label{k4}
\begin{split}
\frac{1}{2}&\frac{d}{dt}\Big(\|r^{(n-1)/2}\ti{u}\|^2
+\bar{u}\|r^{(n-1)/2}v^0\|^2\Big)+\|r^{(n-1)/2}\ti{u}_r\|^2\\
=&-\int_a^b r^{n-1}\ti{u}v^0\ti{u}_r dr\\
\leq & \frac{1}{2}\|r^{(n-1)/2}\ti{u}_r\|^2
+\frac{1}{2}\|\ti{u}\|_{L^\infty}^2\|r^{(n-1)/2}v^0\|^2.
\end{split}
\ee
Note that $\|\ti{u}\|_{L^\infty}$ can be estimated as follows
\ben
|\ti{u}(r,t)|=|u^0(r,t)-\bar{u}|
=\left|\int_a^r u^0_rdr\right|
\leq \left(\int_a^b u^0dr\right)^{1/2}
\left(\int_a^b \frac{(u^0_r)^2}{u^0}dr\right)^{1/2}.
\enn
Then substituting the above estimate into \eqref{k4} and integrating the result over $(0,t)$ we have
\ben
\begin{split}
\frac{1}{2}&\|r^{(n-1)/2}\ti{u}(t)\|^2
+\frac{1}{2}\bar{u}\|r^{(n-1)/2}v^0(t)\|^2
+\frac{1}{2}\int_0^t\|r^{(n-1)/2}\ti{u}_r\|^2\,d\tau\\
\leq& \frac{1}{2}\int_0^t\int_a^b\frac{(u^0_r)^2}{u^0}\,drd\tau\cdot
\|u^0\|_{L^\infty(0,t;L^1)}
\|r^{(n-1)/2}v^0\|^2_{L^\infty(0,t;L^2)},
\end{split}
\enn
which, along with \eqref{k1} and the fact
\ben
\|u^0\|_{L^\infty(0,t;L^1)}\leq C_0
\sup_{\tau\in[0,t]}\left\{\int_a^b r^{n-1}[(u^0\ln u^0-u^0)(\tau)-(\bar{u}\ln \bar{u}-\bar{u})
-\ln \bar{u}(u^0(\tau)-\bar{u})]dr\right\}
\enn
implies \eqref{k2}. The proof is completed.

\end{proof}
We proceed to derive higher regularity properties for the solution $(\tilde{u},v^0)$ of \eqref{e8}.
\begin{lemma}\label{l02}
 Suppose the assumptions in Theorem \ref{p1} hold and $\bar{u}>0$. Let $(\tilde{u},v^0)(r,t)$ be the solution of \eqref{e8}. Then there is a constant $C_0$ independent of $t$ such that
\be\label{k5}
\|(r^{n-1}v^0)_r(t)\|^2+\|r^{(n-1)/2}\ti{u}_r(t)\|^2
+\int_0^t\big(\bar{u}\|(r^{n-1}v^0)_r\|^2
+\|r^{(n-1)/2}\ti{u}_t\|^2\big)d\tau \leq C_0.
\ee
\end{lemma}
\begin{proof}
We multiply the second equation of \eqref{e8} with $r^{n-1}$ and differentiate the resulting equation with respect to $r$. Then from the first equation of \eqref{e8} we obtain
\be\label{e9}
(r^{n-1}v^0)_{rt}=(r^{n-1}\ti{u}_r)_r
=r^{n-1}\ti{u}_t-(r^{n-1}\ti{u}v^0)_r
-\bar{u}(r^{n-1}v^0)_r.
\ee
Taking the $L^2$ inner product of $\eqref{e9}$ against $2(r^{n-1}v^0)_r$ to get
\be\label{k6}
\begin{split}
\frac{d}{dt}&\|(r^{n-1}v^0)_r\|^2+2\bar{u}\|(r^{n-1}v^0)_r\|^2\\
=&2\int_{a}^b r^{n-1}\ti{u}_t (r^{n-1}v^0)_r dr
-2\int_{a}^b (r^{n-1}\ti{u}v^0)_r(r^{n-1}v^0)_r dr\\
:=&I_1+I_2.
\end{split}
\ee
We may rewrite $I_1$ as
\ben
\begin{split}
I_1=2\frac{d}{dt}\int_a^b (r^{n-1}\ti{u})(r^{n-1}v^0)_r dr
-2\int_a^b (r^{n-1}\ti{u})(r^{n-1}v^0)_{rt} dr:=M_1+M_2,
\end{split}
\enn
where $M_1$ can be reorganized as
\ben
M_1
=\frac{d}{dt}\left(\frac{1}{2}\|(r^{n-1}v^0)_r\|^2 +2\|r^{n-1}\ti{u}\|^2
-\left\|\frac{1}{\sqrt{2}}(r^{n-1}v^0)_r
-\sqrt{2}r^{n-1}\ti{u}\right\|^2 \right)
\enn
and $M_2$ can be estimated by \eqref{e9} and the Poincar\'{e} inequality as
\ben
\begin{split}
M_2
&=-2\int_a^b (r^{n-1}\ti{u})(r^{n-1}\ti{u}_r)_{r} dr
&=2\int_a^b (r^{n-1}\ti{u})_r(r^{n-1}\ti{u}_r) dr
\leq C_0 \|r^{(n-1)/2}\ti{u}_r\|^2.
\end{split}
\enn
Hence
\ben
I_1\leq \frac{d}{dt}\left(\frac{1}{2}\|(r^{n-1}v^0)_r\|^2 +2\|r^{n-1}\ti{u}\|^2
-\left\|\frac{1}{\sqrt{2}}(r^{n-1}v^0)_r
-\sqrt{2}r^{n-1}\ti{u}\right\|^2 \right)
+C_0 \|r^{(n-1)/2}\ti{u}_r\|^2.
\enn
To estimate $I_2$, we first note that for fixed $t>0$ if $f(r,t)\in H^1$ satisfies $f|_{r=a,b}=0$ it follows that $
f(r,t)^2=2\int_a^r f f_r dr\leq 2 \|f(t)\| \|f_r(t)\|,
$ which leads to
\be\label{k12}
\|f(t)\|_{L^\infty}\leq \sqrt{2}\|f(t)\|^{1/2} \|f_r(t)\|^{1/2}\quad\,\, {\rm and}\,\,\,\,\|f(t)\|_{L^\infty}\leq C_0 \|f_r(t)\|,
\ee
thanks to the Poincar\'{e} inequality $\|f(t)\|\leq C_0 \|f_r(t)\|$.
Then we deduce from \eqref{k12} and the Sobolev embedding inequality that
\be\label{k8}
\begin{split}
I_2\leq& \frac{\bar{u}}{2} \|(r^{n-1}v^0)_r\|^2
+\frac{4}{\bar{u}}
\|\ti{u}\|_{L^\infty}^2\|(r^{n-1}v^0)_r\|^2+
\frac{4}{\bar{u}}\|\ti{u}_r\|^2
\|r^{n-1}v^0\|_{L^\infty}^2\\
\leq& \frac{\bar{u}}{2} \|(r^{n-1}v^0)_r\|^2
+C_0 \|r^{(n-1)/2}\ti{u}_r\|^2 \big(\|r^{(n-1)/2}v^0\|^2+\|(r^{n-1}v^0)_r\|^2\big).
\end{split}
\ee
Substituting the above estimates for $I_1$ and $I_2$ into \eqref{k6}, one derives
\ben
\begin{split}
\frac{d}{dt}&\left(\frac{1}{2}\|(r^{n-1}v^0)_r\|^2
+\|\frac{1}{\sqrt{2}}(r^{n-1}v^0)_r
-\sqrt{2}r^{n-1}\ti{u}\|^2 \right)
+\frac{3}{2}\bar{u}\|(r^{n-1}v^0)_r\|^2\\
\leq&
C_0 \|r^{(n-1)/2}\ti{u}_r\|^2 \|(r^{n-1}v^0)_r\|^2
+C_0 \|r^{(n-1)/2}\ti{u}_r\|^2 \big(\|r^{(n-1)/2}v^0\|^2+1\big)
+2\frac{d}{dt}\|r^{n-1}\ti{u}\|^2.
\end{split}
\enn
 Then applying Gronwall's inequality to the above result and using Lemma \ref{l01}, we conclude that
\be\label{k7}
\|(r^{n-1}v^0)_r(t)\|^2
+\bar{u}\int_0^t\|(r^{n-1}v^0)_r\|^2
d\tau \leq C_0.
\ee
We proceed to estimate $\|r^{(n-1)/2}\ti{u}_r(t)\|$ by multiplying the first equation of \eqref{e8} with $2r^{n-1}\ti{u}_t$ in $L^2$ and derive
\be\label{k25}
\begin{split}
\frac{d}{dt}\|r^{(n-1)/2}\ti{u}_r\|^2+2\|r^{(n-1)/2}\ti{u}_t\|^2
&=2\int_a^b (r^{n-1}\ti{u}v^0)_r \ti{u}_tdr+2\bar{u}\int_{a}^b (r^{n-1}v^0)_r \ti{u}_tdr\\
&:=I_3+I_4.
\end{split}
\ee
By similar arguments as deriving \eqref{k8}, we estimate $I_3$ as
\ben
\begin{split}
I_3
\leq \frac{1}{2}\|r^{(n-1)/2}\ti{u}_t\|^2
+C_0 \|r^{(n-1)/2}\ti{u}_r\|^2
(\|r^{(n-1)/2}v^0\|^2+\|(r^{n-1}v^0)_r\|^2)
\end{split}
\enn
and by the
Cauchy-Schwarz inequality, $I_4$ is estimated as
\ben
I_4\leq \frac{1}{2}\|r^{(n-1)/2}\ti{u}_t\|^2
+C_0\|(r^{n-1}v^0)_r\|^2.
\enn
Then feeding \eqref{k25} on the above estimates for $I_3$ and $I_4$, we have
 \be\label{i9}
 \begin{split}
 \frac{d}{dt}&\|r^{(n-1)/2}\ti{u}_r\|^2
 +\|r^{(n-1)/2}\ti{u}_t\|^2\\
 \leq& C_0 \|r^{(n-1)/2}\ti{u}_r\|^2
(\|r^{(n-1)/2}v^0\|^2+\|(r^{n-1}v^0)_r\|^2)
+C_0\|(r^{n-1}v^0)_r\|^2.
\end{split}
 \ee
Integrating \eqref{i9} over $(0,t)$ and using \eqref{k2} and \eqref{k7}, one arrives at
\ben
\|r^{(n-1)/2}\ti{u}_r(t)\|^2
+\int_0^t\|r^{(n-1)/2}\ti{u}_t\|^2
d\tau \leq C_0,
\enn
which, in conjunction with \eqref{k7} gives \eqref{k5}. The proof is completed.

\end{proof}
\begin{lemma}\label{l03}
Suppose that the assumptions in Theorem \ref{p1} hold and $\bar{u}>0$. Then there exists a constant $C_0$ independent of $t$ such that
\be\label{k10}
\|r^{(n-1)/2}\ti{u}_t(t)\|^2+\|(r^{n-1}v^0)_{rr}(t)\|^2
+\int_0^t \big(\|r^{(n-1)/2}\ti{u}_{rt}\|^2+\bar{u}\|(r^{n-1}v^0)_{rr}\|^2
\big)d\tau\leq  C_0
\ee
and
\be\label{k11}
\|(r^{n-1}\ti{u}_r)_r(t)\|^2+\int_0^t
\big(\|(r^{n-1}\ti{u}_r)_{r}\|^2+\|(r^{n-1}\ti{u}_r)_{rr}\|^2\big)d\tau
\leq C_0.
\ee
\end{lemma}
\begin{proof}
Differentiating the first equation of \eqref{e8} with respect to $t$ and multiplying the result with $2r^{n-1}\ti{u}_t$, we get upon integration by parts that
\be\label{k13}
\begin{split}
\frac{d}{dt}&\|r^{(n-1)/2}\ti{u}_t\|^2+2\|r^{(n-1)/2}
\ti{u}_{rt}\|^2\\
&=-2\int_a^b (r^{n-1}\ti{u}v^0)_t \ti{u}_{rt}dr
-2\bar{u}\int_a^b (r^{n-1}v^0)_t\ti{u}_{rt}\\
&:=I_5+I_6.
\end{split}
\ee
 By \eqref{k12} and the second equation of \eqref{e8} we have that
\ben
\begin{split}
I_5\leq& C_0\left(\|\ti{u}_t\|_{L^\infty}
\|r^{(n-1)/2}v^0\|\|r^{(n-1)/2}\ti{u}_{rt}\|
+\|\ti{u}\|_{L^\infty}\|v^0_t\|\|r^{(n-1)/2}\ti{u}_{rt}\|
\right)\\
\leq& C_0\left(\|r^{(n-1)/2}\ti{u}_{t}\|^{1/2}
\|r^{(n-1)/2}\ti{u}_{rt}\|^{3/2}\|r^{(n-1)/2}v^0\|
+\|r^{(n-1)/2}\ti{u}_r\| \|\ti{u}_r\|\|r^{(n-1)/2}\ti{u}_{rt}\|
\right)
\\
\leq & \frac{1}{2}\|r^{(n-1)/2}\ti{u}_{rt}\|^2
+C_0\left( \|r^{(n-1)/2}\ti{u}_{t}\|^2\|r^{(n-1)/2}v^0\|^4
+\|r^{(n-1)/2}\ti{u}_r\|^4
\right).
\end{split}
\enn
 We use again the second equation of \eqref{e8} and Cauchy-Schwarz inequality to get
\ben
I_6\leq \frac{1}{2}\|r^{(n-1)/2}\ti{u}_{rt}\|^2
+2\bar{u}^2 \|r^{(n-1)/2}\ti{u}_r\|^2.
\enn
Substituting the above estimates for $I_5$-$I_6$ into \eqref{k13}, then integrating the results over $(0,t)$ and using Lemma \ref{l01} along with Lemma \ref{l02}, we conclude that
\be\label{k14}
\|r^{(n-1)/2}\ti{u}_t(t)\|^2
+\int_0^t \|r^{(n-1)/2}\ti{u}_{rt}\|^2d\tau\leq  C_0.
\ee
We proceed to estimating the remaining part
$\|(r^{n-1}v^0)_{rr}(t)\|^2
+\bar{u}\int_0^t \|(r^{n-1}v^0)_{rr}\|^2 d\tau$ in \eqref{k10}. Differentiating \eqref{e9} with respect to $r$
and multiplying the resulting equation with $2(r^{n-1}v^0)_{rr}$ we get
\be\label{k15}
\begin{split}
\frac{d}{dt}&\|(r^{n-1}v^0)_{rr}\|^2+2\bar{u}\|(r^{n-1}v^0)_{rr}\|^2\\
=&2\int_a^b (r^{n-1}\ti{u}_t)_r(r^{n-1}v^0)_{rr}dr\\
&-2\int_a^b (r^{n-1}\ti{u}v^0)_{rr}(r^{n-1}v^0)_{rr}\\
:=&I_{7}+I_8.
\end{split}
\ee
To estimate of $I_7$,
we note for $g(r,t)\in L^2(a,b)$ with fixed $t>0$, it follows that
\be\label{k24}
b^{-(n-1)}\|r^{(n-1)/2}g(t)\|^2\leq\|g(t)\|^2\leq a^{-(n-1)}\|r^{(n-1)/2}g(t)\|^2.
\ee
Then from Cauchy-Schwarz inequality and \eqref{k24} one derives
\ben
I_7\leq \frac{\bar{u}}{2}\|(r^{n-1}v^0)_{rr}\|^2
+C_0(\|r^{(n-1)/2}\ti{u}_t\|^2+\|r^{(n-1)/2}\ti{u}_{rt}\|^2).
\enn
To bound $I_8$ we first estimate $\int_0^t\|(r^{n-1}\ti{u}_r)_r\|^2d\tau$ by the first equation of \eqref{e8} as follows:
\be\label{k16}
\begin{split}
\int_0^t\|(r^{n-1}\ti{u}_r)_r\|^2d\tau
\leq& \int_0^t\|r^{n-1}\ti{u}_t\|^2 d\tau
+C_0\int_0^t \|\ti{u}_{r}\|^2d\tau \cdot \|(r^{n-1}v^0)\|_{L^\infty(0,t;H^1)}^2\\
&+C_0\bar{u}^2\int_{0}^t\|(r^{n-1}v^0)_r\|^2d\tau\\
\leq& C_0,
\end{split}
\ee
where \eqref{k12} and Lemma \ref{l01} - Lemma \ref{l02} have been used. Then \eqref{k16} along with \eqref{k24} and \eqref{k2} implies that
\be\label{k17}
\int_0^t\|\ti{u}_{rr}\|^2d\tau\leq C_0\int_0^t(\|(r^{n-1}\ti{u}_r)_r\|^2+\|r^{(n-1)/2}
\ti{u}_r\|^2)d\tau\leq C_0,
 \ee
where the constant $C_0$ depends on $a$ and $b$. Noting that $(r^{n-1}\ti{u}v^0)_{rr}=
(r^{n-1}v^0)_{rr}\ti{u}
+2(r^{n-1}v^0)_r\ti{u}_r+(r^{n-1}v^0)\ti{u}_{rr}$, one deduces by \eqref{k12} and the Sobolev embedding inequality that
\ben
\begin{split}
I_8\leq& \frac{\bar{u}}{2}\|(r^{n-1}v^0)_{rr}\|^2
+\frac{2}{\bar{u}}\|(r^{n-1}\ti{u}v^0)_{rr}\|^2\\
\leq &\frac{\bar{u}}{2}\|(r^{n-1}v^0)_{rr}\|^2
+C_0\big(\|\ti{u}_r\|^2\|(r^{n-1}v^0)_{rr}\|^2\\
&+\|\ti{u}_r\|^2\|(r^{n-1}v^0)_{r}\|^2
+ \|\ti{u}_{rr}\|^2\|(r^{n-1}v^0)_{r}\|^2
+\|\ti{u}_{rr}\|^2\|v^0\|^2\big).
\end{split}
\enn
We feed \eqref{k15} on the above estimates for $I_7$-$I_8$ then apply Gronwall's inequality, Lemma \ref{l01} - Lemma \ref{l02}, \eqref{k14} and \eqref{k17} to the result to find
\ben
\|(r^{n-1}v^0)_{rr}(t)\|^2
+\bar{u}\int_0^t \|(r^{n-1}v^0)_{rr}\|^2
d\tau\leq  C_0,
\enn
which, along with \eqref{k14} yields \eqref{k10}. We next prove \eqref{k11}.
By similar arguments as deriving \eqref{k16} one gets
\be\label{k20}
\begin{split}
\|(r^{n-1}\ti{u}_r)_r(t)\|^2
\leq &\|r^{n-1}\ti{u}_t(t)\|^2
+ C_0\|\ti{u}_{r}(t)\|^2 \|(r^{n-1}v^0)(t)\|_{H^1}^2\\
&+C_0\bar{u}^2\|(r^{n-1}v^0)_r(t)\|^2\\
\leq& C_0,
\end{split}
\ee
where \eqref{k14} and Lemma \ref{l01} - Lemma \ref{l02} have been used. We differentiate \eqref{e9} with respect to $r$ and conclude that
\be\label{k19}
\begin{split}
\int_0^t&\|(r^{n-1}\ti{u}_r)_{rr}\|^2d\tau\\
\leq &C_0\left(\int_0^t\|r^{n-1}\ti{u}_{rt}\|^2 d\tau
+\int_0^t\|r^{(n-1)/2}\ti{u}_{t}\|^2 d\tau
+\bar{u}^2\int_0^t \|(r^{n-1}v^0)_{rr}\|^2d\tau
\right)
\\
&+C_0\int_0^t \big(\|\ti{u}_{r}\|^2+\|\ti{u}_{rr}\|^2\big)d\tau \cdot \|(r^{n-1}v^0)\|_{L^\infty(0,t;H^2)}^2
\\
\leq &C_0,
\end{split}
\ee
where we have used \eqref{k10}, Lemma \ref{l01} and Lemma \ref{l02}. Finally collecting \eqref{k16}, \eqref{k20} and \eqref{k19} we derive \eqref{k11}. The proof is finished.

\end{proof}
We are now in the position to prove Theorem \ref{p1} by the above Lemma \ref{l01} - Lemma \ref{l03}.\\
\newline
\textbf{Proof of Theorem \ref{p1}}.
We first prove Part (i) of Theorem \ref{p1}. By Lemma \ref{l01} and \eqref{k24}, one derives
\be\label{k21}
\|v^0(t)\|^2\leq C_0\|r^{(n-1)/2}v^0(t)\|^2\leq C_0,\qquad
\|\ti{u}(t)\|^2+\int_0^t \|\ti{u}\|_{H^1}^2d\tau\leq C_0,
\ee
where the constant $C_0$ depends on $a,b$ and $n$ and the Poincar\'{e} inequality $\|\ti{u}\|^2\leq C_0\|\ti{u}_r\|^2$ has been used.
 On the other hand, for $f(r,t)\in H^1$ with fixed $t$ we have
 \be\label{jj1}
 \begin{split}
 \|f_r\|^2=
&\|r^{-(n-1)}[(r^{n-1}f)_r-(n-1)r^{n-2}f]\|^2\\
\leq &a^{-2(n-2)}\|(r^{n-1}f)_r\|^2
+a^{-2(n-1)}(n-1)b^{2(n-2)}\|f\|^2\\
\leq & C_0 (\|(r^{n-1}f)_r\|^2+\|f\|^2).
\end{split}
\ee
  Then it follows from Lemma \ref{l02}, Lemma \ref{l01}, \eqref{k24} and \eqref{jj1} that
\be\label{k22}
\|v^0_r(t)\|^2+\|\ti{u}_r(t)\|^2
+\int_0^t\big(\bar{u}\|(r^{(n-1)}v^0)_r\|^2
+\|\ti{u}_t\|^2\big)d\tau\leq C_0.
\ee
Similarly, it follows from Lemma \ref{l03} and \eqref{jj1} that
\be\label{jj2}
\|v^0_{rr}(t)\|^2+\|\tilde{u}_{rr}(t)\|^2
+\int_0^t\left(\bar{u}\|(r^{(n-1)}v^0)_{rr}\|^2
+\|\tilde{u}_{rr}\|^2+\|\tilde{u}_{rrr}\|^2\right)\,d\tau
\leq C_0.
\ee
 Thus collecting \eqref{k21}, \eqref{k22} and \eqref{jj2} we derive the desired \emph{a priori} estimate \eqref{k23}, which along with the fixed point theorem implies the existence of solution $(u^0,v^0)$ in $C([0,\infty);H^2\times H^2)$.

 We next prove \eqref{i10}. Integrating \eqref{i9} over $(0,\infty)$ with respect to $t$, then using Lemma \ref{l01} and Lemma \ref{l02}, we have
 \ben
 \begin{split}
 \int_0^\infty& \frac{d}{dt}\|r^{(n-1)/2}\ti{u}_r\|^2dt\\
 \leq &C_0 \|r^{(n-1)/2}\ti{u}_r\|_{L^2(0,\infty;L^2)}^2
 \left(\|r^{(n-1)/2}v^0\|_{L^\infty(0,\infty;L^2)}^2
 +\|(r^{n-1}v^0)_r\|_{L^\infty(0,\infty;L^2)}^2\right)\\
 &+C_0\|(r^{n-1}v^0)_r\|_{L^2(0,\infty;L^2)}^2\\
 \leq& C_0,
 \end{split}
 \enn
 which, along with \eqref{k2} implies that
 $
 \|r^{(n-1)/2}\ti{u}_r\|^2\in W^{1,1}(0,\infty).
 $
  Hence, it follows that
  \ben
  \lim_{t\rightarrow \infty}\|\ti{u}_r\|
  \leq C_0\lim_{t\rightarrow \infty}\|r^{(n-1)/2}\ti{u}_r\|=0,
  \enn
  which, along with the Gagliardo-Nirenberg inequality  $\|(u^0-\bar{u})(t)\|_{L^\infty}^2\leq C_0 \|(u^0-\bar{u})(t)\|_{L^2}
  \|(u^0-\bar{u})_r(t)\|_{L^2}$ and \eqref{k2}, gives \eqref{i10}.
  Part (i) of Theorem \ref{p1} is thus proved.

We proceed to prove Part (ii). When $\bar{u}=0$, for $0<T<\infty$ one can easily deduce the \emph{a priori} estimates  \eqref{k26} by the standard energy method that bootstraps the regularity of the solution $(u^0,v^0)$ from $L^2$ to $H^2$. We omit this procedure for simplicity and refer readers to \cite{LZ} for details. Then the existence of solution $(u^0,v^0)$ follows from \eqref{k26} and the fixed point theorem. The proof is finished.

\endProof


\section{Proof of Theorem \ref{t1} and Proposition \ref{p2}.}
 Let $(u^\va,v^\va)$ and $(u^0,v^0)$ be the solutions of \eqref{e3}-\eqref{e4} corresponding to $\va>0$ and $\va=0$ respectively. Then the initial-boundary value problem for their differences $h:=u^\va-u^0$, $w:=v^\va-v^0$ reads:
\be\label{e5}
\left\{\begin{split}
&h_{t}=\frac{1}{r^{n-1}}(r^{n-1}h_r)_r+\frac{1}{r^{n-1}}(r^{n-1}hw)_r
+\frac{1}{r^{n-1}}(r^{n-1}u^0w)_r+\frac{1}{r^{n-1}}(r^{n-1}hv^0)_r,\\
&w_{t}=\va\frac{1}{r^{n-1}}(r^{n-1}w_{r})_r-2\va ww_r+h_r
+\va\frac{1}{r^{n-1}}(r^{n-1}v^0_{r})_r-2\va (wv^0_r+ v^0w_r
+v^0v^0_r)\\
&\quad\quad\,-\va \frac{n-1}{r^2}(w+v^0),\qquad \quad(r,t)\in (a,b)\times (0,\infty) \\
&(h,w)(r,0)=(0,0),\\
&h|_{r=a,b}=0,\, \,\, w|_{r=a}=\bar{v}_1-v^0(a,t),\,w|_{r=b}=\bar{v}_2-v^0(b,t).
\end{split}
\right.
\ee

To prove Theorem \ref{t1} we shall invoke an elementary result (see Lemma \ref{l0}) on an ordinary differential equation (ODE) and a series of lemmas on the \emph{a priori} estimates for solutions of \eqref{e5}. In particular, the $L^2$-estimate for solution $(h,w)$ and higher regularity estimates for the solution component $h$ will be established in Lemma \ref{l1} - Lemma \ref{l4}, and Lemma \ref{l5} will give a weighted $L^2$-estimate for the derivative of $w$.

We proceed to prove the following Lemma, which gives an upper bound for the solution of an ODE involving a small parameter $\gamma$. It extends a result in \cite{constantin1986note,wuxu2014} with $k=2$ to any integer $k\geq 2$.

\begin{lemma}\label{l0}
 Let $k\geq 2$ be an integer and $0<T<\infty$. Let $C_0>1$ be a constant independent of $T$ and $f_{1}(t)$, $f_{2}(t)\geq 0$ be two continuous functions on $[0,T]$. Consider the ODE
\be\label{e0}
\left\{
\aligned
&\frac{d}{dt}y(t)\leq \gamma f_{1}(t)+f_{2}(t)y(t)+C_0[y^2(t)+\cdots+y^{k}(t)],\\
&y(0)=0.
\endaligned
\right.
\ee
If we set
\be\label{j1}
\gamma_{0}=\min\left\{
[4(k-1)]^{-1}\left(\int_{0}^Tf_{1}(t)\,dt\right)^{-1},\,\,\,\,
[8TG(k-1)^2]^{-1}\left(\int_{0}^{T}
f_{1}(t)\,dt\right)^{-1}
\right\},
\ee
with $G:=C_0\left(e^{\int_{0}^T f_{2}(t)\,dt}
\right)^{k-1}$.
Then for $\gamma\in (0,\gamma_{0}]$, any solution $y(t)\geq 0$ of \eqref{e0} satisfies
\be\label{j5}
y(t)\leq e^{\int_{0}^{T}f_{2}(t)\,dt}\cdot\min\left\{3,\,\,
\frac{3}{2T(k-1)G},\,\,12(k-1)\gamma\int_{0}^{T}f_{1}(t)\,dt
\right\}, \quad t\in [0,T].
\ee
\end{lemma}

\begin{proof}
Let $U(t)=y(t)e^{-\int_{0}^t f_{2}(\tau)\,d\tau}$. Then \eqref{e0} can be rewritten as
\ben
\frac{d}{dt}U(t)\leq
\gamma f_{1}(t)e^{-\int_{0}^t f_{2}(\tau)\,d\tau}+C_0\left(e^{\int_{0}^t f_{2}(\tau)\,d\tau}
\right)U^2+\cdots+C_0\left(e^{\int_{0}^t f_{2}(\tau)\,d\tau}
\right)^{k-1}U^k.
\enn
Noting that $e^{\int_0^tf_2(\tau)d\tau}\geq 1$ thanks to $f_2(t)\geq 0$,
 we deduce that
\be\label{i06}
\left\{
\aligned
&\frac{d}{dt}U(t)\leq \gamma f_{1}(t)+GU^2(t)(1+U(t))^{k-2},\\
&U(0)=0.
\endaligned
\right.
\ee
For later use, we define
\be\label{j2}
\sigma =\min\left\{
G,\,\,\frac{1}{4T^2(k-1)^2G},\,\,
16(k-1)^2\gamma^2G\left(\int_{0}^T f_{1}(t)\,dt\right)^2
\right\}.
\ee
Now dividing both sides of \eqref{i06} by $\left(1+\sqrt{\frac{G}{\sigma}}\right)^k$, it follows that
\ben
\begin{split}
\frac{\frac{d}{dt}U(t)}{\left(1+\sqrt{\frac{G}{\sigma}}U(t)\right)^k}
\leq \gamma f_{1}(t)+\frac{G U^2(t)}{\left(1+\sqrt{\frac{G}{\sigma}}U(t)\right)^2}
\cdot \frac{\left(1+U(t)\right)^{k-2}}
{\left(1+\sqrt{\frac{G}{\sigma}}U(t)\right)^{k-2}}.
\end{split}
\enn
 Then noting $\sigma\leq G$ due to its definition \eqref{j2}, we deduce from the above inequality that
\ben
\frac{\frac{d}{dt}U(t)}{\left(1+\sqrt{\frac{G}{\sigma}}U(t)\right)^k}
\leq &\gamma f_{1}(t) +\sigma,
\enn
which integrated over $(0,t)$ with $t\in (0,T]$,  yields
\be\label{j3}
\begin{split}
\frac{\sqrt{\frac{\sigma}{G}}}{k-1}\cdot \frac{1}
{\left(1+\sqrt{\frac{G}{\sigma}}U(t)\right)^{k-1}}
\geq&  \frac{\sqrt{\frac{\sigma}{G}}}{k-1}
-\sigma T-\gamma \int_{0}^{T}f_{1}(t)\,dt\\
\geq&  \frac{\sqrt{\frac{\sigma}{G}}}{2(k-1)}
-\gamma \int_{0}^{T}f_{1}(t)\,dt,
\end{split}
\ee
where we have used the fact $\sigma T\leq \frac{\sqrt{\frac{\sigma}{G}}}{2(k-1)}$, thanks to the definition of $\sigma$.
We shall prove that
\be\label{j4}
\gamma \int_{0}^{T}f_{1}(t)\,dt\leq
\frac{\sqrt{\frac{\sigma}{G}}}{4(k-1)},
\ee
of which the proof is split into three cases by the value of $\sigma$.\\
Case 1, when $\sigma=G$, it follows from the definition of $\gamma_0$ in \eqref{j1} that
\ben
\gamma\int_{0}^{T}f_{1}(t)\,dt\leq \gamma_{0}
\int_{0}^{T}f_{1}(t)\,dt\leq\frac{1}{4(k-1)}=
\frac{\sqrt{\frac{\sigma}{G}}}{4(k-1)}.
\enn
Case 2, when $\sigma=\frac{1}{4T^2(k-1)^2G}$, we have by using \eqref{j1} again that
\ben
\gamma\int_{0}^{T}f_{1}(t)\,dt\leq \gamma_{0}
\int_{0}^{T}f_{1}(t)\,dt
\leq \frac{1}{8TG(k-1)^2}
=
\frac{\sqrt{\frac{\sigma}{G}}}{4(k-1)}.
\enn
Case 3, when $\sigma=16(k-1)^2\gamma^2G\left(\int_{0}^T f_{1}(t)\,dt\right)^2$, one immediately get
\ben
\gamma\int_{0}^{T}f_{1}(t)\,dt=
\frac{\sqrt{\frac{\sigma}{G}}}{4(k-1)}.
\enn
Hence combining the above Case 1 - Case 3, we conclude that \eqref{j4} holds true and it follows from \eqref{j4} and \eqref{j3} that
\ben
\left(1+\sqrt{\frac{G}{\sigma}}U(t)\right)^{k-1}
\leq  4, \quad \,t\in [0,T]
\enn
thus
\ben
U(t)\leq 3 \sqrt{\frac{\sigma}{G}},\quad \,t\in [0,T]
\enn
 which, along with \eqref{j2} and the definition of $U(t)$, yields the desired estimate \eqref{j5}.
 The proof is finished.

\end{proof}

In the sequel, for convenience we denote
\ben
E(t):=\|r^{(n-1)/2}h(t)\|^2+\|r^{(n -1)/2}w(t)\|^2
+\va\|r^{(n-1)/2}w_{r}(t)\|^2,
\enn
\be\label{j14}
F(t):=\|u^0(t)\|_{H^2}^2+\|v^0(t)\|_{H^2}+
\|v^0(t)\|_{H^2}^2+\|v^0(t)\|_{H^2}^4+|\bar{v}_1|^2+|\bar{v}_2|^2+1.
\ee
The following lemma gives the $L^2$-estimate for the solution $(h,w)$ of problem \eqref{e5}.
\begin{lemma}\label{l1}
Let $0<t<\infty$. Then there exists a constant $C_0$ independent of $\va$ and $t$, such that
\be\label{i01}
\begin{split}
\frac{d}{dt}&(\|r^{(n-1)/2}h(t)\|^2+\|r^{(n-1)/2}w(t)\|^2)
+\frac{3}{2}\|r^{(n-1)/2}h_r(t)\|^2\\
&+2\va \|r^{(n-1)/2}w_r(t)\|^2+2(n-1)\va \|r^{(n-3)/2}w(t)\|^2\\
\leq& C_0\va^2 F(t)+C_0F(t)E(t)+C_0E^2(t)+C_0E^3(t)+2\va[r^{n-1}w_{r}w]|_a^b.
\end{split}
\ee
\end{lemma}
\begin{proof}
 Testing the first equation of \eqref{e5} with $2r^{n-1}h$ in $L^2$ and using integration by parts, we get
\ben
\begin{split}
\frac{d}{dt}\|r^{(n-1)/2}h\|^2+2\|r^{(n-1)/2}h_r\|^2
&=-2\int_{a}^br^{n-1}hwh_r\,dr-2\int_a^b r^{n-1}(u^0w+hv^0)h_r\,dr\\
&:=J_1+J_2.
\end{split}
\enn
 The estimate of $J_1$ follows from \eqref{k12} and \eqref{k24}:
\ben
\begin{split}
J_1\leq &2 \|r^{(n-1)/2}h_r\| \|h\|_{L^\infty}
\|r^{(n-1)/2}w\|\\
\leq &C_0 \|r^{(n-1)/2}h_r\|^{\frac{3}{2}} \|r^{(n-1)/2}h\|^{\frac{1}{2}}
\|r^{(n-1)/2}w\|\\
\leq& \frac{1}{8}\|r^{(n-1)/2}h_r\|^2
+C_0 \|r^{(n-1)/2}h\|^2\|r^{(n-1)/2}w\|^4.
\end{split}
\enn
On the other hand, the Sobolev embedding inequality and Cauchy-Schwarz inequality entail that
\ben
J_2\leq \frac{1}{8}\|r^{(n-1)/2}h_r\|^2
+C_0\|u^0\|_{H^1}^2\|r^{(n-1)/2}w\|^2
+C_0\|v^0\|_{H^1}^2\|r^{(n-1)/2}h\|^2.
\enn
Collecting the above estimates for $J_1$ and $J_2$, we conclude that
\be\label{i1}
\begin{split}
\frac{d}{dt}\|r^{(n-1)/2}h(t)\|^2
+\frac{7}{4}\|r^{(n-1)/2}h_r(t)\|^2
\leq C_0 F(t)E(t) +C_0E^3(t).
\end{split}
\ee
We proceed by taking the $L^2$ inner product of the second equation of \eqref{e5} with $2r^{n-1}w$ and get
\ben
\begin{split}
\frac{d}{dt}&\|r^{(n-1)/2}w\|^2+2\va\|r^{(n-1)/2}w_r\|^2
+2(n-1)\va \|r^{(n-3)/2}w\|^2\\
=& \,2\va [r^{n-1}w_rw]|_a^b
-4\va \int_a^b r^{n-1}(ww_r+w v^0_r)w dr\\
&+2\int_a^b \big(r^{n-1}h_r + \va(r^{n-1}v^0_r)_r\big)w dr\\
&
-2\va \int_a^b \big(2r^{n-1}v^0w_r +2 r^{n-1}v^0v^0_r
+(n-1)r^{n-3}v^0\big)w dr\\
:=\,& 2\va [r^{n-1}w_rw]|_a^b+J_3+J_4+J_5.
\end{split}
\enn
First Sobolev embedding inequality and \eqref{k24} yield
\ben
\begin{split}
J_3\leq &4\va \|w\|_{L^\infty}\|r^{(n-1)/2}w_r\|\|r^{(n-1)/2}w\|
+ 4\va \|v^0_r\|_{L^\infty}\|r^{(n-1)/2}w\|^2\\
\leq& C_0\va \big(\|r^{(n-1)/2}w_r\|+\|r^{(n-1)/2}w\|\big)
\|r^{(n-1)/2}w_r\|\|r^{(n-1)/2}w\|
+C_0\va \|v^0\|_{H^2}\|r^{(n-1)/2}w\|^2.
\end{split}
\enn
It follows from Cauchy-Schwarz inequality and \eqref{k24} that
\ben
J_4\leq \frac{1}{4}\|r^{(n-1)/2}h_r\|^2
+C_0\|r^{(n-1)/2}w\|^2+C_0\va^2 \|v^0\|_{H^2}^2.
\enn
 Moreover Sobolev embedding inequality, Cauchy-Schwarz inequality and \eqref{k24} lead to
\ben
J_5\leq \va \|r^{(n-1)/2}w_r\|^2
+C_0\va \|v^0\|_{H^1}^2\|r^{(n-1)/2}w\|^2
+\|r^{(n-1)/2}w\|^2
+C_0\va^2\|v^0\|_{H^2}^4+C_0\va^2\|v^0\|^2.
\enn
Collecting the above estimates for $J_3$-$J_5$ and recalling that $0<\va<1$, we end up with
\ben
\begin{split}
\frac{d}{dt}&\|r^{(n-1)/2}w(t)\|^2+2\va\|r^{(n-1)/2}w_r(t)\|^2
+2(n-1)\va \|r^{(n-3)/2}w(t)\|^2\\
\leq&2\va [r^{n-1}w_rw]|_a^b+ \frac{1}{4}\|r^{(n-1)/2}h_r(t)\|^2
+C_0\va^2 F(t)
+C_0F(t)E(t)+C_0E^2(t),
\end{split}
\enn
which, adding to \eqref{i1} gives the desired estimate \eqref{i01}. The proof is completed.

\end{proof}

We turn to estimate the derivative of $w$ and the boundary term in \eqref{i01}.
\begin{lemma}\label{l2} Let $0<t<\infty$. Then there exists a constant $C_0$ independent of $\va$ and $t$, such that
\be\label{i02}
\begin{split}
\frac{d}{dt}&\big(\va \|r^{(n-1)/2}w_r(t)\|^2\big)+\frac{1}{2}
\|r^{(n-1)/2}w_t(t)\|^2\\
&\leq C_0 \va^2 F(t)+C_0F(t)E(t)+C_0 E^2(t)
+\|r^{(n-1)/2}h_r(t)\|^2
+2\va[r^{n-1}w_rw_t]|_a^b
\end{split}
\ee
and
\be\label{i03}
\begin{split}
2\va&[r^{n-1}w_rw]|_a^b+2\va[r^{n-1}w_rw_t]|_a^b\\
&\leq C_0 \va^{1/2}F(t)
+C_0 F(t)E(t)+C_0 E^2(t)
+\frac{1}{4}\|r^{(n-1)/2}w_t(t)\|^2
+\frac{1}{4}\|r^{(n-1)/2}h_r(t)\|^2.
\end{split}
\ee
\end{lemma}
\begin{proof}
Taking the $L^2$ inner product of the second equation of \eqref{e5} with $2r^{n-1}w_t$, then using integration by parts we have
\ben
\begin{split}
\frac{d}{dt}&\va \|r^{(n-1)/2}w_r(t)\|^2+2
\|r^{(n-1)/2}w_t(t)\|^2\\
=&2\va[r^{n-1}w_rw_t]|_a^b-4\va\int_a^b r^{n-1} ww_r w_t dr
-4\va \int_a^b r^{n-1}(wv^0_r +v^0 w_r+v^0 v^0_r) w_t dr\\
&+2 \int_a^b \big[r^{n-1}h_r+\va(r^{n-1}v^0_r)_r
-\va(n-1)r^{n-3}w
-\va(n-1)r^{n-3}v_0 \big]w_t dr\\
:=&2\va[r^{n-1}w_rw_t]|_a^b+J_6+J_7+J_8.
\end{split}
\enn
We first employ Sobolev embedding inequality and \eqref{k24} to deduce that
\ben
\begin{split}
J_6\leq& C_0\va\|w\|_{H^1}\|r^{(n-1)/2}w_r\|  \|r^{(n-1)/2}w_t\|\\
\leq &\frac{1}{8}\|r^{(n-1)/2}w_t\|^2
+C_0\va^2\big(\|r^{(n-1)/2}w_r\|^2
+\|r^{(n-1)/2}w\|^2
\big)\|r^{(n-1)/2}w_r\|^2
\end{split}
\enn
and that
\ben
J_7\leq \frac{1}{8}\|r^{(n-1)/2}w_t\|^2
+C_0\va^2\big( \|v^0\|_{H^2}^2\|r^{(n-1)/2}w\|^2
+\|v^0\|_{H^1}^2\|r^{(n-1)/2}w_r\|^2
+ \|v^0\|_{H^1}^4\big).
\enn
Moreover Cauchy-Schwarz inequality and \eqref{k24} entail that
\ben
J_8
\leq \frac{5}{4}\|r^{(n-1)/2}w_t\|^2+\|r^{(n-1)/2}h_r\|^2
+C_0\va^2(
\|v^0\|_{H^2}^2+\|r^{(n-1)/2}w\|^2).
\enn
Then \eqref{i02} follows from the above estimates on $J_6$-$J_8$. It remains to prove \eqref{i03}.
By the definition of $w$ and Gagliardo-Nirenberg interpolation inequality, one deduces that
\be\label{i2}
\begin{split}
2\va[r^{n-1}w_rw]|_a^b\leq &C_0\va \|w_r\|_{L^\infty}(|\bar{v}_1|+|\bar{v}_2|
+\|v^0\|_{L^\infty})\\
\leq& C_0 \va (\|w_r\|^{\frac{1}{2}}
\|w_{rr}\|^{\frac{1}{2}}
+\|w_r\|)(|\bar{v}_1|+|\bar{v}_2|
+\|v^0\|_{H^1})\\
\leq &\eta \va^2 \|w_{rr}\|^2+C_0(1+1/\eta)\va \|w_{r}\|^2+C_0(\va^{1/2}+\va)
(|\bar{v}_1|+|\bar{v}_2|
+\|v^0\|_{H^1})^2,
\end{split}
\ee
where $\eta>0$ is a small constant to be determined. By a similar argument as deriving \eqref{i2} and the second equation of \eqref{e3} with $\va=0$, we further get
\be\label{i3}
\begin{split}
2\va[r^{n-1}w_rw_t]|_a^b
\leq &\eta \va^2 \|w_{rr}\|^2+C_0(1+1/\eta)\va \|w_{r}\|^2+C_0(\va^{1/2}+\va)
\|v^0_{t}\|_{H^1}^2\\
\leq &\eta \va^2 \|w_{rr}\|^2+C_0(1+1/\eta)\va \|w_{r}\|^2+C_0(\va^{1/2}+\va)
\|u^0\|_{H^2}^2.
\end{split}
\ee
To bound the term $\|w_{rr}\|^2$ in \eqref{i2} and \eqref{i3}, we use the second equation of \eqref{e5}, Sobolev embedding inequality and \eqref{k24} and derive
\be\label{i4}
\begin{split}
\va^2 \|w_{rr}\|^2\leq C_1&
\big(\|r^{(n-1)/2}w_t\|^2+\|r^{(n-1)/2}h_r\|^2
+\va^2 \|r^{(n-1)/2}w_r\|^2
+\va^2\|r^{(n-1)/2}w\|^2
\|v^0\|^2_{H^2}
\\
&+\va^2 \|r^{(n-1)/2}w_r\|^4
+\va^2 \|r^{(n-1)/2}w\|^2\|r^{(n-1)/2}w_r\|^2
+\va^2\|v^0\|^2_{H^2}\\
&+\va^2\|r^{(n-1)/2}w_r\|^2
\|v^0\|^2_{H^1}
+\va^2 \|v^0\|^4_{H^1}
+\va^2 \|r^{(n-1)/2}w\|^2\big),
\end{split}
\ee
where we have used the notation $C_1$ to distinguish it from the constant $C_0$ in \eqref{i2} and \eqref{i3}.
Finally feeding \eqref{i2} and \eqref{i3} on \eqref{i4} then adding the results, we obtain \eqref{i03} by taking $\eta$ small enough such that $C_1\eta<\frac{1}{8}$ and by using $0<\va<1$. The proof is completed.

\end{proof}

 We next apply Lemma \ref{l0} to the combination of Lemma \ref{l1} and Lemma \ref{l2} to obtain the following result.
\begin{lemma}\label{l3}  Let $0<T<\infty$ and $\va\in (0,\va_0]$ with $\va_0$ defined in Theorem \ref{t1}. Then there exists a constant $C$ independent of $\va$, depending on $T$ such that
\be\label{i04}
\begin{split}
\|&h\|_{L^\infty(0,T;L^2)}^2+
\|w\|_{L^\infty(0,T;L^2)}^2\\ &+
\va \|w_r\|_{L^\infty(0,T;L^2)}^2+
\|h_r\|_{L^2(0,T;L^2)}^2+
\|w_t\|_{L^2(0,T;L^2)}^2
\leq C\va^{\frac{1}{2}}
\end{split}
\ee
and
\be\label{jj4}
\|w_{rr}\|_{L^2(0,T;L^2)}\leq C\va^{-3/4}.
\ee
\end{lemma}
\begin{proof}
We first add \eqref{i02} and \eqref{i03} to \eqref{i01} and find
\be\label{i5}
\begin{split}
\frac{d}{dt}&E(t)
+\frac{1}{4}\|r^{(n-1)/2}h_r(t)\|^2
+2(n-1)\va \|r^{(n-3)/2}w(t)\|^2
+\frac{1}{4} \|r^{(n-1)/2}w_t(t)\|^2\\
&\leq C_0\va^{\frac{1}{2}}F(t)+C_0F(t)E(t)
+C_0E^2(t)+C_0E^3(t),
\end{split}
\ee
where $0<\va<1$ has been used. Then we apply lemma \ref{l0} to \eqref{i5} by taking $k=3,\,\gamma=\va^{1/2}$ and $f_1(t)=f_2(t)=C_0F(t)$ to conclude for $\va\in (0,\va_0]$ that
\be\label{j12}
\|h\|_{L^\infty(0,T;L^2)}^2
+\|w\|_{L^\infty(0,T;L^2)}^2
+\va \|w_r\|_{L^\infty(0,T;L^2)}^2
\leq \big(
C_0e^{C_0\int_{0}^{T}F(t)\,dt}\int_{0}^{T}F(t)dt
\big)\va^{\frac{1}{2}},
\ee
where \eqref{k24} has been used.
Then we integrate \eqref{i5} over (0,T) and applying \eqref{j12} to the result to deduce that
\ben
 \|h\|_{L^\infty(0,T;L^2)}^2+
\|w\|_{L^\infty(0,T;L^2)}^2+
\va \|w_r\|_{L^\infty(0,T;L^2)}^2+
\|h_r\|_{L^2(0,T;L^2)}^2+
\|w_t\|_{L^2(0,T;L^2)}^2
\leq C\va^{\frac{1}{2}},
 \enn
  where the constant $C$ depending on $T$ and $\int_0^{T}F(t)dt$ is finite thanks to Theorem \ref{p1}. We thus derive \eqref{i04} and proceed to prove \eqref{jj4}. Indeed, it follows from the second equation of \eqref{e5}, Sobolev embedding inequality and \eqref{i04} that
  \ben
  \begin{aligned}
  \va\|w_{rr}\|_{L^2(0,T;L^2)}
  \leq& C\left(
  \|w_t\|_{L^2(0,T;L^2)}+\va\|w_r\|_{L^2(0,T;L^2)}
  +\va \|w\|_{L^\infty(0,T;H^1)}\|w_r\|_{L^2(0,T;L^2)}
  \right)\\
  &+C\left(\|h_r\|_{L^2(0,T;L^2)}
  +\va \|v^0\|_{L^2(0,T;H^2)}
  +\va \|w_r\|_{L^2(0,T;L^2)}\|v^0\|_{L^\infty(0,T;H^2)}
  \right)\\
  &+C\left(\va\|v^0\|_{L^2(0,T;H^2)}
  +\va\|w\|_{L^2(0,T;L^2)}
  +\va \|v^0\|_{L^2(0,T;L^2)}
  \right)\\
  \leq& C\va^{1/4}.
  \end{aligned}
  \enn
  We thus derive \eqref{jj4} and the proof is completed.

\end{proof}

Higher regularity estimates for the solution component $h$ is given in the following lemma.
\begin{lemma}\label{l4}
 Suppose $0<T<\infty$ and $\va\in (0,\va_0]$. Then there is a constant $C$ independent of $\va$, depending on $T$ such that
\be\label{i05}
\|h_{r}\|_{L^\infty(0,T;L^2)}^2
+\|h_{t}\|_{L^\infty(0,T;L^2)}^2
+\|h_{rt}\|_{L^2(0,T;L^2)}^2
\leq C\va^{1/2}.
\ee
\end{lemma}
\begin{proof}
We first take the $L^2$ inner product of the first equation of \eqref{e5} with $2r^{n-1}h_t$ and use integration by parts to get
\be\label{i6}
\begin{split}
\frac{d}{dt}&\|r^{(n-1)/2}h_r\|^2
+2\|r^{(n-1)/2}h_t\|^2\\
=&-2\int_a^b r^{n-1}(hw +u^0 w +hv^0 )h_{rt}dr\\
\leq & \,\frac{1}{2}\|r^{(n-1)/2}h_{rt}\|^2
+C_0(\|h_r\|^2\|w\|^2+\|u^0\|_{H^1}^2\|w\|^2
+\|h_r\|^2\|v^0\|^2).
\end{split}
\ee
Then differentiating the first equation of \eqref{e5} with respect to $t$ and multiplying the resulting equation with $2r^{n-1}h_t$ in $L^2$, we derive
\ben
\begin{split}
\frac{d}{dt}&\|r^{(n-1)/2}h_t\|^2
+2\|r^{(n-1)/2}h_{rt}\|^2\\
=&-2\int_a^b r^{n-1}h_t w h_{rt}dr
-2\int_a^b r^{n-1}\big(h w_t
+u^0_t w
+u^0 w_t
+h_t v^0
+h v^0_t\big) h_{rt}dr\\
:=&K_1+K_2.
\end{split}
\enn
The estimate for $K_1$ follows from \eqref{k12} and \eqref{k24}:
\ben
\begin{split}
K_1
\leq&C_0\|r^{(n-1)/2}h_{rt}\|
\|h_{t}\|_{L^\infty}
\|w\|\\
\leq& C_0 \|r^{(n-1)/2}h_{rt}\|^{3/2}
\|r^{(n-1)/2}h_{t}\|^{1/2}
\|w\|\\
\leq& \frac{1}{4}\|r^{(n-1)/2}h_{rt}\|^2
+C_0 \|r^{(n-1)/2}h_{t}\|^2
\|w\|^4.
\end{split}
\enn
Sobolev embedding inequality and \eqref{k24} entail that
\ben
\begin{split}
 K_2 \leq &
\frac{1}{4}\|r^{(n-1)/2}h_{rt}\|^2
+C_0\Big(\|r^{(n-1)/2}h_{r}\|^2
\|w_{t}\|^2
+\|u^0_t\|_{H^1}^2\|w\|^2\Big)\\
&+C_0\Big(\|u^0\|_{H^1}^2\|w_t\|^2
+\|v^0\|_{H^1}^2\|r^{(n-1)/2}h_t\|^2
+\|v^0_t\|_{H^1}^2\|h\|^2\Big).
\end{split}
\enn
Then collecting the above estimates for $K_1$ and $K_2$, one derives
\be\label{i7}
\begin{split}
\frac{d}{dt}&\|r^{(n-1)/2}h_t\|^2
+\frac{3}{2}\|r^{(n-1)/2}h_{rt}\|^2\\
\leq& C_0\left(\|w\|^4+\|v^0\|_{H^1}^2\right)
\|r^{(n-1)/2}h_t\|^2
+C_0\|w_t\|^2\|r^{(n-1)/2}h_r\|^2\\
&+C_0\|u^0\|_{H^1}^2\|w_t\|^2
+C_0 \left(\|u^0\|_{H^3}^2+\|u^0\|_{H^2}^2
\|v^0\|_{H^2}^2\right)\|w\|^2
+C_0\|u^0\|_{H^2}^2
\|h\|^2,
\end{split}
\ee
where we have used inequalities $\|u^0_t\|_{H^1}^2\leq C_0 (\|u^0\|_{H^3}^2+\|u^0\|_{H^2}^2
\|v^0\|_{H^2}^2)$ and $\|v^0_t\|_{H^1}^2\leq C_0 \|u^0\|_{H^2}^2$, thanks to the first and second equations of \eqref{e7}.
Finally by adding \eqref{i7} to \eqref{i6} and applying Gronwall's inequality to the result, and then using \eqref{i04} we obtain \eqref{i05}. The proof is completed.

\end{proof}

We turn to establish a weighted $L^2$-estimate (enlightened by \cite{JZ}) on the derivative of $w$.
\begin{lemma}\label{l5}
For $0<T<\infty$ and $\va\in (0,\va_0]$, there is a constant $C$ independent of $\va$, depending on $T$ such that
\ben
\|(r-a)(r-b)w_r\|^2_{L^\infty(0,T;L^2)}+\va \|(r-a)(r-b)w_{rr}\|^2_{L^2(0,T;L^2)}
\leq C\va^{1/2}.
\enn
\end{lemma}
\begin{proof}
 Taking the $L^2$ inner product of the second equation of \eqref{e5} with $-2(r-a)^2(r-b)^2 w_{rr}$ and using integration by parts, one gets
 \be\label{i8}
 \begin{split}
 \frac{d}{dt}&\|(r-a)(r-b)w_r\|^2
 +2\va \|(r-a)(r-b)w_{rr}\|^2\\
 =\,&-2\va\int_a^b (r-a)^2(r-b)^2w_{rr}\Big[\frac{1}{r^{n-1}}(r^{n-1}v^0_{r})_r
 -\frac{n-1}{r^2}(w
 +v_0) \Big]dr\\
 &+4\va\int_a^b (r-a)^2(r-b)^2w_{rr}(wv^0_{r}
 +v^0w_r
 +v^0v^0_r)dr\\
 &-4 \int_a^b(2r-a-b) (r-a)(r-b)w_{r}w_t dr\\
 &-2\int_a^b (r-a)^2(r-b)^2w_{rr} h_r dr\\
 &-2\va \int_a^b (r-a)^2(r-b)^2w_{rr}\Big(\frac{n-1}{r}w_r-2 ww_r\Big) dr\\
 :=& \sum_{i=3}^{7}K_{i}.
 \end{split}
 \ee
 We next estimate $K_3$-$K_{7}$. Indeed Cauchy-Schwarz inequality, Sobolev embedding inequality and \eqref{k24} yield
 \ben
 \begin{split}
 K_3+K_4\leq& \frac{1}{8}
 \va\|(r-a)(r-b)w_{rr}\|^2\\
 &+C_0\va (\|v^0\|_{H^2}^2+\|w\|^2+\|v^0\|_{H^2}^2\|w\|^2
 +\|v^0\|_{H^1}^2\|w_r\|^2+\|v^0\|_{H^1}^4
 )
 \end{split}
 \enn
 and
 \ben
 K_5+K_7\leq \frac{1}{8}\va\|(r-a)(r-b)w_{rr}\|^2
 +C_{0}(1+\va+\va \|w\|^2+\va \|w_r\|^2)\|(r-a)(r-b)w_{r}\|^2
 +C_{0}\|w_{t}\|^2.
 \enn
 For the term $K_6$, we use integration by parts and the first equation of \eqref{e5} to get
 \ben
 \begin{split}
 K_6=\,&
 4\int_a^b(2r-a-b)(r-a)(r-b)w_rh_r dr
 +2\int_a^b (r-a)^2(r-b)^2 w_r h_{rr}dr\\
 =\,&4\int_a^b(2r-a-b)(r-a)(r-b)w_rh_r dr
 +2\int_a^b(r-a)^2(r-b)^2w_r\Big(h_t-\frac{n-1}{r}h_r\Big) dr\\
 &-2\int_a^b(r-a)^2(r-b)^2w_r(hw+u^0w)_r dr\\
 &-2\int_a^b(r-a)^2(r-b)^2w_r\Big[(hv^0)_r
 +\frac{n-1}{r}(hw +u^0w+hv^0)\Big] dr
 \\
 :=&R_1+R_2+R_3+R_4.
 \end{split}
 \enn
 We proceed to estimate $R_1$-$R_4$.
 First it follows from Cauchy-Schwarz inequality that
 \ben
 R_1+R_2\leq \|(r-a)(r-b)w_r\|^2
 +C_0 (\|h_r\|^2+\|h_t\|^2).
 \enn
 Moveover, we use Cauchy-Schwarz inequality and apply \eqref{k12} to $h$ and $(r-a)(r-b)w$ to derive
 \ben
 \begin{split}
 R_3\leq& 2(\|h\|_{L^\infty}+\|u^0\|_{L^\infty})
 \|(r-a)(r-b)w_r\|^2\\
 &+2(\|h_r\|+\|u^0_r\|)
 \|(r-a)(r-b)w_r\|
 \|(r-a)(r-b)w\|_{L^\infty}\\
 \leq &
 C_0(\|h_r\|+\|u^0\|_{H^1}) \|(r-a)(r-b)w_r\|^2\\
 &+C_0(\|h_r\|+\|u^0_r\|)
 \|(r-a)(r-b)w_r\|\|[(r-a)(r-b)w]_r\|\\
 \leq &C_0(\|h_r\|+\|u^0\|_{H^1}+1)^2 \|(r-a)(r-b)w_r\|^2
 +C_0 \|w\|^2.
 \end{split}
 \enn
  The estimate for $R_4$ follows from the Sobolev embedding inequality, \eqref{k12} and Cauchy-Schwarz inequality:
 \ben
 \begin{split}
 R_4
 \leq &C_0 \|(r-a)(r-b)w_r\|
 (\|hw\|+\|hv^0\|_{H^1}+\|u^0w\|)\\
 \leq &\|(r-a)(r-b)w_r\|^2
 +C_0(\|h_r\|^2 \|w\|^2
 +\|h_r\|^2\|v^0\|_{H^1}^2
 +\|u^0\|_{H^1}^2\|w\|^2
 ).
 \end{split}
 \enn
We thus conclude from the above estimates for $R_1$-$R_4$ that
\ben
\begin{split}
K_6\leq& C_0(\|h_r\|+\|u^0\|_{H^1}+1)^2\|(r-a)(r-b)w_r\|^2\\
&+C_0(\|h_r\|^2+\|h_t\|^2+\|h_r\|^2\|w\|^2
+\|u^0\|_{H^1}^2\|w\|^2+\|h_r\|^2\|v^0\|_{H^1}^2+\|w\|^2).
\end{split}
\enn
Substituting the above estimates for $K_3$-$K_7$ into \eqref{i8}, then applying Gronwall's inequality, Lemma \ref{l3}, Lemma \ref{l4} and Theorem \ref{p1} to the result, we obtain the desired estimate and the proof is finished.

\end{proof}

We next prove Theorem \ref{t1} by the results derived in Lemma \ref{l3} - Lemma \ref{l5}.\\
~\\
\textbf{Proof of Theorem \ref{t1}}.
By Lemma \ref{l3}, Lemma \ref{l4} and Sobolev embedding inequality, we deduce that
\ben
\|u^\va-u^0\|_{L^\infty(0,T;C[a,b])}\leq C_0
\|u^\va-u^0\|_{L^\infty(0,T;H^1)}
\leq C\va^{1/4},
\enn
which gives \eqref{j6}. Clearly $\delta^2\leq \frac{4}{(b-a)^2}(r-a)^2(r-b)^2$ holds for $\delta<\frac{b+a}{2}$ and $r\in (a,b)$, thus it follows from Lemma \ref{l5} that
\ben
\delta^2\int_{a+\delta}^{b-\delta}
w^2_r(r,t)\,dr
\leq \frac{4}{(b-a)^2}
\int_{a+\delta}^{b-\delta}(r-a)^2(r-b)^2 w^2_{r}(r,t)\,dr
\leq C\va^{1/2},\quad t\in [0,T]
\enn
which, along with Lemma \ref{l3} and Gagliardo-Nirenberg inequality entails that
\ben
\begin{split}
\|v^\va-v^0\|_{L^\infty(0,T;C[a+\delta,b-\delta])}
\leq& C_0 \big(\|w\|_{L^\infty(0,T;L^2(a+\delta,b-\delta))}\\
&+\|w\|_{L^\infty(0,T;L^2(a+\delta,b-\delta))}^{1/2}
\|w_r\|_{L^\infty(0,T;L^2(a+\delta,b-\delta))}^{1/2}
\big)\\
\leq &C\big(\va^{1/4}+\va^{1/8}\cdot \va^{1/8}\delta^{-1/2}\big)\\
\leq& C\va^{1/4}\delta^{-1/2},
\end{split}
\enn
provided $\delta<1$.
Hence we derive \eqref{j7} and we next prove the equivalence between \eqref{j8} and \eqref{j9}. We first prove that \eqref{j9} implies \eqref{j8}. Assume $\int_0^{t_0} u^0_r(a,\tau)\,d\tau\neq 0$ for some $t_0\in [0,T]$.  Then integrating the second equation of \eqref{e3} with $\va=0$ over $(0,t_0)$ along with compatible condition $v_{0}(a)=\bar{v}_1$ gives
\be\label{j10}
v^0(a,t_0)=\bar{v}_1+\int_0^{t_0}u^0_r(a,\tau)\,d\tau.
\ee
We thus have
\ben
\begin{split}
\liminf_{\va\rightarrow 0} \|v^\va-v^0\|_{L^\infty(0,T;C[a,b])}
\geq \liminf_{\va \rightarrow 0} |\bar{v}_1-v^0(a,t_0)|
=\liminf_{\va \rightarrow 0}|\int_0^{t_0}u^0_r(a,\tau)\,d\tau|
>0.
\end{split}
\enn
Similar arguments lead to \eqref{j8} when assuming that
$\int_0^{t_0} u^0_r(b,\tau)\,d\tau\neq 0$ for some $t_0\in [0,T]$. We thus proved \eqref{j8} provided \eqref{j9}. The proof of that $\eqref{j8}$ indicates $\eqref{j9}$ will follow from the argument of contradiction. Indeed, if we assume that \eqref{j8} holds and the opposite of \eqref{j9} holds, that is
\ben
\int_0^t u^0_r(a,\tau)\,d\tau=0\quad {\rm and}\quad \int_0^{t} u^0_r(b,\tau)\,d\tau= 0,\quad {\rm for }\,\,{\rm all}\,\, t\in [0,T],
\enn
which, along with the second equation of \eqref{e3} with $\va=0$ leads to
\ben
\begin{split}
(v^{\va}-v^0)(a,t)
=-[v^0(a,t)-\bar{v}_1]
=-\int_0^tu^0_r(a,\tau)d\tau=0,\\
(v^{\va}-v^0)(b,t)
=-[v^0(b,t)-\bar{v}_2]
=-\int_0^tu^0_r(b,\tau)d\tau=0,
\end{split}
\enn
where the compatible conditions $v_0(a)=\bar{v}_1, \,v_0(b)=\bar{v}_2$ have been used.
Thus $w|_{r=a,b}=(v^\va-v^0)|_{r=a,b}=0$ and $w_t|_{r=a,b}=[(v^\va-v^0)|_{r=a,b}]_t=0$ and the terms $2\va [r^{n-1}w_rw]|_{a}^b$, $2\va [r^{n-1}w_rw_t]|_{a}^b$ in \eqref{i01} and \eqref{i02} would vanish. Then by similar arguments as deriving \eqref{j12}, we conclude that
\ben
\|h\|_{L^\infty(0,T;L^2)}^2
+\|w\|_{L^\infty(0,T;L^2)}^2
+\va \|w_r\|_{L^\infty(0,T;L^2)}^2
\leq C\va^{2},
\enn
which, along with Sobolev embedding inequality gives rise to
\ben
\lim_{\va\rightarrow 0}\|v^\va-v^0\|_{L^\infty(0,T;C[a,b])}
\leq C_0\lim_{\va\rightarrow 0}(\|w\|_{L^\infty(0,T;L^2)}
+\|w_r\|_{L^\infty(0,T;L^2)})=0,
\enn
which, contradicts with \eqref{j8}, thus \eqref{j8} implies \eqref{j9}.
The proof is completed.

\endProof
 \\~
We next convert the result of Theorem \ref{t1} to the initial-boundary value problem \eqref{e6} for the original chemotaxis model to prove Proposition \ref{p2}.
\\
\\~
\textbf{Proof of Proposition \ref{p2}.}
 Let
 \be\label{i14}
 \begin{split}
 &c^{\va}(r,t)=c_0(r)\exp\left\{
 \int_0^t
 \left[-u^\va+\va (v^{\va})^2-\va v^\va_r-\va\frac{n-1}{r}v^\va\right](r,\tau)
 d\tau
 \right\},\\
 &c^0(r,t)=c_0(r)\exp\left\{-
 \int_0^t u^0(r,\tau)d\tau
 \right\},
 \end{split}
 \ee
 where $(u^\va,v^{\va})$ and $(u^0,v^0)$ are the solutions of problem \eqref{e3}-\eqref{e4} with $\va>0$ and $\va=0$ respectively. It is easy to check that $(u^\va,c^\va)$ and $(u^0,c^0)$ with $c^\va$ and $c^0$ defined \eqref{i14} are the unique solutions of  \eqref{e6} corresponding to $\va>0$ and $\va=0$, respectively.

The first inequality in \eqref{i13} follows directly from the Sobolev embedding inequality, Lemma \ref{l3} and Lemma \ref{l4} as following:
\ben
\|u^\va-u^0\|_{L^\infty(0,T;C[a,b])}
\leq C_0 \|u^\va-u^0\|_{L^\infty(0,T;H^1)}
\leq C\va^{1/4}.
\enn
To prove the second inequality in \eqref{i13}, we deduce from \eqref{i14} that
\be\label{a1}
\begin{aligned}
|c^{\va}&(r,t)-c^{0}(r,t)|\\
=&|c^{0}(r,t)|\cdot\bigg|
\exp\bigg\{\int_{0}^{t}[-(u^{\va}-u^{0})
+\va(v^{\va})^{2}-\va v^{\va}_{r}-\va\frac{n-1}{r}v^\va]\,d\tau\bigg\}-1\bigg|\\
=&|c^{0}(r,t)|\cdot|e^{G^{\va}(r,t)}-1|,
\end{aligned}
\ee
where we have denoted by $G^{\va}(r,t)=\int_{0}^{t}[-(u^{\va}-u^{0})
+\va(v^{\va})^{2}-\va v^{\va}_{r}-\va\frac{n-1}{r}v^\va]\,d\tau$ for convenience.
For $(r,t)\in [a,b]\times[0,T]$, it follows from Lemma \ref{l3} and the Sobolev embedding inequality that
\be\label{jj5}
\begin{aligned}
&|G^\va(r,t)|\\
 \leq& C_0T^{1/2}\|u^\va-u^0\|_{L^2(0,T;L^\infty)}
 + C_0T\va \|v^\va\|_{L^\infty(0,T;L^\infty)}^2
 +C_0T^{1/2}\va \|v^\va_r\|_{L^2(0,T;L^\infty)}\\
 &+C_0T^{1/2}\va \|v^\va\|_{L^2(0,T;L^\infty)}
 \\
 \leq &C_0T^{1/2}\|u^\va-u^0\|_{L^2(0,T;H^1)}
 + C_0T\va\|v^\va\|_{L^\infty(0,T;H^1)}^2
 +C_0T^{1/2}\va \|v^\va_r\|_{L^2(0,T;H^1)}\\
 &+C_0T^{1/2}\va \|v^\va\|_{L^2(0,T;H^1)}
 \\
 \leq &C
  \va^{1/4}+
  C\va\left(\|v^0\|_{L^\infty(0,T;H^1)}
  ^2+\va^{-1/2}\right)
 +C \va\left(\|v^0_r\|_{L^2(0,T;H^1)}+\va^{-3/4}\right)\\
 &+ C\va\left(\|v^0\|_{L^2(0,T;H^1)}+\va^{-1/4}\right)
 \\
 \leq& C_1\va^{1/4},
 \end{aligned}
\ee
for some constant $C_1$ independent of $\va$ (depending on $T$) and the assumption $0<\va<1$ has been used in the last inequality. Then we apply Taylor expansion to $e^{G^{\va}(r,t)}$ and using \eqref{jj5} to conclude that
\be\label{a2}
\left|e^{G^{\va}(r,t)}-1\right|\leq \sum_{k=1}^\infty \frac{1}{k!}|G^{\va}(x,t)|^k
\leq \sum_{k=1}^\infty\frac{C_1^k}{k!} \va^{1/4}  \leq C_2\va^{1/4},
\ee
where the assumption $0<\va<1$ has been used in the second inequality and the constant $C_2:=e^{C_1}$ is independent of $\va$.
On the other hand, by the assumptions $c_0(r)>0$ and $\ln c_0\in H^3$ in Proposition \ref{p2} we derive that $\|\ln c_0\|_{L^\infty}\leq \|\ln c_0\|_{H^1}\leq C_3$ for some positive constant $C_3$, which along with the fact $c_0(r)=e^{\ln c_{0}(r)}$ leads to
\be\label{jj6}
e^{-C_3}\leq c_0(r)\leq e^{C_3}\qquad \text{for}\,\,
r\in [a,b].
\ee
Moreover, from Theorem \ref{p1} we know that
\be\label{a3}
\left\|\int_0^t u^0(r,\tau) d\tau
\right\|_{L^\infty(0,T;L^\infty)}\leq C_0 T \|u^0\|_{L^\infty(0,T;H^1)}\leq C_4\qquad \text{for}\,\,t\in [0,T],
\ee
where the constant $C_4$ depending on $T$. Thus we deduce from the second equality of \eqref{i14}, \eqref{jj6} and \eqref{a3} that
\be\label{a4}
0<C_5^{-1}<c^{0}(r,t)<C_5\qquad \text{for}\,\,
(r,t)\in [a,b]\times [0,T],
\ee
with $C_5=e^{(C_3+C_4)}$.
Hence, it follows from \eqref{a1}, \eqref{a2} and \eqref{a4} that
\be\label{a5}
\|c^\va-c^0\|_{L^\infty(0,T;L^\infty)}\leq C_6\va^{1/4},
\ee
where $C_6:=C_2C_5$ is independent of $\va$.
We thus derive the second inequality in \eqref{i13} and proceed to prove \eqref{i11}.
It follows from transformation \eqref{e2} that
\be\label{i17}
c^\va_r-c^0_r=(v^0-v^\va)c^\va +v^0(c^0-c^\va),
\ee
which, in conjunction with \eqref{j7} and \eqref{i13} leads to
\ben
\begin{split}
\|c^\va_r-c^0_r\|_{L^\infty(0,T;C[a+\delta,b-\delta])}
\leq &
\|v^\va-v^0\|_{L^\infty(0,T;C[a+\delta,b-\delta])}
(\|c^0\|_{L^\infty(0,T;C[a,b])}+C\va^{1/4})\\
&+C\va^{1/4}\|v^0\|_{L^\infty(0,T;C[a,b])} \\
\leq &C\va^{1/4}\delta^{-1/2},
\end{split}
\enn
where $\delta<1$ has been used thanks to $\delta(\va)\rightarrow 0$ as $\va\rightarrow 0$. We thus derived \eqref{i11}. To prove the equivalence between \eqref{i12} and \eqref{j9}, we first derives two positive constants $C_7$ and $C_8$ independent of $\va$ such that
\be\label{jj3}
0<C_7\leq c^\va(r,t)\leq C_8 \qquad \text{ for}\quad (r,t)\in [a,b]\times [0,T],
\ee
by choosing $\va$ small enough such that $C_6\va^{1/4}<\frac{1}{2C_5}$ in \eqref{a5} and using \eqref{a4}.
 With \eqref{jj3} in hand, we next prove the equivalence between \eqref{i12} and \eqref{j9}.
First, it follows from \eqref{i17}, \eqref{i13} and \eqref{jj3} that
\be\label{i18}
\begin{split}
&\liminf_{\va\rightarrow 0}
\|c^\va_r-c^0_r\|_{L^\infty(0,T;C[a,b])}\\
\geq&
 \liminf_{\va\rightarrow 0}\left[\|c^\va\|_{L^\infty(0,T;L^\infty)}  \|v^\va-v^0\|_{L^\infty(0,T;L^\infty)}
 -\|v^0\|_{L^\infty(0,T;L^\infty)}  \|c^\va-c^0\|_{L^\infty(0,T;L^\infty)}\right]
\\
 \geq& \liminf_{\va\rightarrow 0}\left[\|c^\va\|_{L^\infty(0,T;L^\infty)}  \|v^\va-v^0\|_{L^\infty(0,T;L^\infty)}
 \right]\\
 &-\limsup_{\va\rightarrow 0}\left[\|v^0\|_{L^\infty(0,T;L^\infty)}  \|c^\va-c^0\|_{L^\infty(0,T;L^\infty)}
 \right]\\
 \geq & C_7 \liminf_{\va\rightarrow 0} \|v^\va-v^0\|_{L^\infty(0,T;L^\infty)}.
 \end{split}
\ee
Dividing \eqref{i17} by $c^\va$ and applying a similar argument as deriving \eqref{i18}, then using \eqref{jj3} and \eqref{i13}, one deduces that
\ben
\begin{split}
&\liminf_{\va\rightarrow 0} \|v^\va-v^0\|_{L^\infty(0,T;C[a,b])}\\
\geq&
\liminf_{\va\rightarrow 0} \left\|\frac{c^\va_r-c^0_r}{c^\va}\right\|_{L^\infty
(0,T;L^\infty)}-\limsup_{\va\rightarrow 0} \left\|v^0\left(1-\frac{c^0}{c^\va}\right)
\right\|_{L^\infty(0,T;L^\infty)}\\
\geq & \frac{1}{C_8} \,\liminf_{\va\rightarrow 0}\|c^\va_r-c^0_r\|_{L^\infty
(0,T;L^\infty)},
\end{split}
\enn
which, in conjunction with \eqref{i18} indicates the equivalence between \eqref{i12} and \eqref{j8}. Then we conclude that \eqref{i12} is equivalent to \eqref{j9} by using Theorem \ref{t1}. The proof is completed.

\endProof

\section*{Acknowledgements}
This work is supported by China Postdoctoral Science Foundation
(No.2019M651269), National Natural Science Foundation of China (No.11901139).

\setlength{\bibsep}{0.5ex}
\bibliography{rf}
\bibliographystyle{plain}

\end{document}